\documentclass[10pt,reqno,a4paper]{amsart}
\usepackage{amscd,amssymb,palatino}
\usepackage[utf8]{inputenc}
\usepackage[english]{babel}
\usepackage{caption}
\usepackage{etex}
\usepackage[leqno]{amsmath}
\usepackage{amssymb}
\usepackage{color}
\usepackage{shadow}
\usepackage{epsfig}
\usepackage{epic}
\usepackage[dvipsnames]{xcolor}
\usepackage{graphics}
\usepackage{graphicx}
\usepackage{psfrag}
\usepackage{url}
\usepackage{hyperref}
\usepackage{calc}
\usepackage{tikz}
\usepackage{tikz-cd}
\usepackage[dvipsnames,prologue,table]{pstricks}
\usepackage{pstricks}
\usetikzlibrary{arrows,decorations,patterns,positioning,automata,shadows,fit,shapes,calc,decorations.markings,backgrounds,scopes,decorations.text}

\usepackage{fancyhdr}
\fancypagestyle{plain}{\fancyhf{}
\fancyfoot[C]{{\large\sf\bfseries\thepage}}}
\usepackage{calc}
\newcommand\1{\lower 9pt\hbox{\underbar{}}}
\numberwithin{equation}{section}

\newtheorem {theorem} {Theorem} [section]

\newtheorem {lemma}[theorem]{Lemma}
\newtheorem {prop}[theorem]     {Proposition}

\newtheorem {Corollary}  [theorem]               {Corollary}

\theoremstyle{definition}
\newtheorem {defi}[theorem]{Definition}
\newtheorem {Remark}[theorem]          {Remark}

\newtheorem* {theorem*}{Theorem}

\newcommand{\pr} {\smallskip\noindent{\bf Proof\,\,}}

\def\R{\mathbb{R}}

\def\C{\mathbb{C}}

\newcommand{\Op}{\mathcal{O}p}

\DeclareMathOperator\Div{div}
\DeclareMathOperator\curl{curl}

\title{Universality of Euler flows and flexibility of Reeb embeddings}

\author{Robert Cardona}\address{ Robert Cardona,
Laboratory of Geometry and Dynamical Systems, Department of Mathematics, Universitat Polit\`{e}cnica de Catalunya and BGSMath Barcelona Graduate School of
Mathematics,  Avinguda del Doctor Mara\~{n}on 44-50, 08028 , Barcelona  \it{e-mail: robert.cardona@upc.edu }
 }
 \thanks{Robert Cardona acknowledges financial support from the Spanish Ministry of Economy and Competitiveness, through the Mar\'ia de Maeztu Programme for Units of Excellence in R\& D (MDM-2014-0445) via an FPI grant.}
\author{Eva Miranda}\address{ Eva Miranda,
Laboratory of Geometry and Dynamical Systems, Department of Mathematics-IMTech, EPSEB, Universitat Polit\`{e}cnica de Catalunya and
\\ Centre de Recerca Matem\`{a}tica, Campus UAB Edifici C, 08193 Bellaterra, Barcelona
 \it{e-mail: eva.miranda@upc.edu}
 }
\thanks{Robert Cardona and Eva Miranda are supported by the grant PID2019-103849GB-I00 of MCIN/ AEI /10.13039/501100011033 and AGAUR grant 2021SGR00603. Eva Miranda  is supported by the Catalan Institution for Research and Advanced Studies via an ICREA Academia Prize 2016 and by the Spanish State
Research Agency, through the Severo Ochoa and Mar\'{\i}a de Maeztu Program for Centers and Units
of Excellence in R\&D (project CEX2020-001084-M) }
\author{Daniel Peralta-Salas} \address{Daniel Peralta-Salas, Instituto de Ciencias Matem\'aticas-ICMAT, C/ Nicol\'{a}s Cabrera, nº 13-15 Campus de Cantoblanco, Universidad Aut\'{o}noma de Madrid,
28049 Madrid, Spain \it{e-mail: dperalta@icmat.es} }
\thanks{Daniel Peralta-Salas is supported by the grant PID2019-106715GB GB-C21 funded by MCIN/AEI/10.13039/501100011033. Robert Cardona, Eva Miranda and Daniel Peralta-Salas acknowledge partial support from the grant ``Computational, dynamical
and geometrical complexity in fluid dynamics'', Ayudas Fundaci\'on BBVA a Proyectos de Investigaci\'on Cient\'ifica 2021.}
\author{Francisco Presas} \address{Francisco Presas, Instituto de Ciencias Matem\'aticas-ICMAT, C/ Nicol\'{a}s Cabrera, nº 13-15 Campus de Cantoblanco, Universidad Aut\'{o}noma de Madrid,
28049 Madrid, Spain \it{e-mail: fpresas@icmat.es} }
\thanks{Francisco Presas is supported by the grant MTM2016-79400-P. Daniel Peralta-Salas and Francisco Presas also acknowledge partial support from the ICMAT-Severo Ochoa grant CEX2019-000904-S}

\begin{document}

\begin{abstract}

The dynamics of an inviscid and incompressible fluid flow on a Riemannian manifold is governed by the Euler equations. Recently, Tao~\cite{T2,T3} launched a programme to address the global existence problem for the Euler and Navier Stokes equations based on the concept of universality. Inspired by this proposal, in this article we prove that the stationary Euler equations exhibit several universality features. More precisely, we show that any non-autonomous flow on a compact manifold can be extended to a smooth stationary solution of the Euler equations on some Riemannian manifold of possibly higher dimension. The solutions we construct are of Beltrami type, and being stationary they exist for all time. Using this result, we establish the Turing completeness of the steady Euler flows, i.e., there exist solutions that encode a universal Turing machine and, in particular, these solutions have undecidable trajectories. Our proofs deepen the correspondence between contact topology and hydrodynamics, which is key to establish the universality of the Reeb flows and their Beltrami counterparts. An essential ingredient in the proofs, of interest in itself, is a novel flexibility theorem for embeddings in Reeb dynamics in terms of an $h$-principle in contact geometry, which unveils the flexible behavior of the steady Euler flows. These results can be viewed as lending support to the intuition that solutions to the Euler equations can be extremely complicated in nature.
\end{abstract}

\maketitle

\section{Introduction}\label{S1}

The dynamics of an inviscid and incompressible fluid flow on a Riemannian manifold $(M,g)$ is described by the Euler equations
\begin{equation}
\partial_t u+\nabla_uu=-\nabla p\,, \qquad \Div u=0\,,
\end{equation}
where $u$ is the velocity field of the fluid, and $p$ is the pressure function. Here $\nabla_u$ is the covariant derivative along $u$ and the differential operators $\nabla$ and $\Div$ are computed using the Riemannian metric $g$.

The analysis of the evolution $u(\cdot,t)$ of a smooth initial condition $u(\cdot,0):=u_0(\cdot)$ is a notoriously difficult problem where even the existence of a global-time solution is a challenging open question (the celebrated blow-up problem for the Euler equations). Recently, Terry Tao launched a programme to address the global existence problem, not only for the Euler equations, but also for their viscid counterpart, i.e. the Navier-Stokes equations, based on the concept of \emph{universality}~\cite{T1,T2,T3}. This notion concerns the Euler equations without fixing neither the ambient manifold $M$ nor the metric $g$, and roughly speaking can be defined as the property that any smooth non-autonomous flow on a manifold $N$ may be ``extended'' to a solution of the Euler equations for some $(M,g)$, where the dimension of $M$ is usually much bigger than the dimension of $N$. In~\cite{T3}, Tao introduced a particular way of extending a smooth (non-autonomous) flow on $N$ to a solution of the Euler equations on a manifold $M$ which is a product $M=N\times \mathbb T^m$ endowed with a warped product metric $g$. In particular, he showed that the set of flows that are extendible in the aforementioned sense is the countable union of nowhere dense sets (in the smooth topology), and that there exists a somewhere dense set of flows that can be extended provided that $N$ is diffeomorphic to the $n$-torus, $n\geq 2$.
This interesting result provides further evidence of the universality of the Euler dynamics, but leaves open the problem whether the Euler equations on some high-dimensional Riemannian manifold can encode the behavior of a universal Turing machine~\cite{T1,T2}. Tao discussed in~\cite{T0,TNat} that the ``Turing completeness'' of the Euler equations could be used as a route to construct solutions of the Navier-Stokes equations that blow-up in finite time, by creating an initial datum that is ``programmed'' to evolve to a rescaled version of itself (as a Von Neumann self-replicating machine).

Motivated by Tao's proposal, our goal in this paper is to address the study of the universality of the Euler equations using stationary solutions, which model fluid flows in equilibrium. While at first glance it seems that the steady Euler flows are too restrictive to encode arbitrarily complicated dynamics, we shall see that the surprising connection between the Euler equations and contact topology, allows us to use the flexibility provided by the existence of $h$-principles in the contact realm to show that the stationary solutions exhibit universality features, and in particular they are Turing complete. It is worth noting that topological, rather than dynamical, universality properties of steady Euler flows were established recently in \cite{C1}.

To this end, we introduce the concept of \emph{Eulerisable flow}~\cite{P1}: a volume-preserving (autonomous) vector field $u$ on $M$
is Eulerisable if there exists a Riemannian metric $g$ on $M$ compatible with the volume form, such that $u$ satisfies the stationary Euler equations on $(M,g)$
\begin{equation}\label{euler_stat}
\nabla_uu=-\nabla p\,, \qquad \Div u=0\,.
\end{equation}
When the dimension of $M$ is odd, a particularly relevant class of Eulerisable fields are those which are proportional to their curl through a not necessarily constant factor (a definition of the curl of a vector field in dimension $n>3$, which is a nonlinear differential operator which assigns to a vector field another vector field, will be introduced in Section~\ref{S2}). These vector fields are known as \emph{Beltrami flows}, and in recent years they have found application as powerful tools to analyze different features of fluid flows, including anomalous weak solutions~\cite{LS09}, complicated vortex structures~\cite{EP1,EP2} and reconnections in Navier-Stokes~\cite{ELP}. The geometric content of the Beltrami fields was unveiled in~\cite{EG,Re}, where connections with Reeb fields of a contact structure and with geodesible flows were established. This remarkable connection, which will be exploited in this paper, allows one to bring tools from (high dimensional) contact topology to the analysis of the stationary Euler equations provided that the Riemannian metric is not fixed, which is precisely the context where Tao introduced the notion of universality.

To state our main results, we need to provide a geometric definition of extendibility.
The following is inspired by Tao's definition in~\cite{T3} but it is different in two aspects. First, the ambient manifold is not a product $N\times \mathbb T^m$, but a high-dimensional sphere, and the metric is not a warped product. Second, since we deal with stationary (i.e., time-independent) solutions of the Euler equations, the time coordinate of the original non-autonomous field becomes one of the spatial local coordinates of the Eulerisable flow.

\begin{defi}\label{def1}
A non-autonomous time-periodic vector field $u_0(\cdot,t)$ on a compact manifold $N$ is \emph{Euler-extendible} if there exists an embedding $e:N\times \mathbb S^1 \rightarrow \mathbb S^n$ for some dimension $n>\text{dim }N+1$ (that only depends on the dimension of $N$), and an Eulerisable flow $u$ on $\mathbb S^n$, such that $e(N\times\mathbb S^1)$ is an invariant submanifold of $u$ and $e_*(u_0(\cdot,\theta)+\partial_\theta)=u$, $\theta\in\mathbb S^1$. If the non-autonomous field $u_0(\cdot,t)$ is not time-periodic, we say it is Euler-extendible if there exists a proper embedding $e:N\times \mathbb R \rightarrow \mathbb R^n$ for some dimension $n>\text{dim }N+1$ (that only depends on the dimension of $N$), and an Eulerisable flow $u$ on $\mathbb R^n$, such that $e(N\times\mathbb R)$ is an invariant submanifold of $u$ and $e_*(u_0(\cdot,\theta)+\partial_\theta)=u$, $\theta\in\mathbb \R$.
\end{defi}

{ {Next, we introduce the following} notion of universality.}

\begin{defi}\label{def:universalEuler}
{ We say that the stationary Euler flows are \emph{universal} if any non-autonomous dynamics $u_0(\cdot,t)$ (on any ambient space) is Euler-extendible.}
\end{defi}

\begin{Remark}
In the time-periodic case, the choice of the ambient manifold $\mathbb S^n$, where the Eulerisable flow $u$ is defined, is made for the sake of concreteness, but all the results we state in this paper hold for any other manifold. However, for general non-autonomous dynamics, the ambient space where $u$ is defined does not need to be $\mathbb R^n$, but it must be noncompact (because we need to embed properly $N\times\mathbb R$).
\end{Remark}

Roughly speaking, the extendibility of a non-autonomous dynamics implies that, in the appropriate local coordinates, $u_0$ describes the ``horizontal'' behavior of the integral curves of the extended vector field $u$. We want to emphasize that $u_0$ is not assumed to be volume-preserving, although certainly $u$ will be.

We are now ready to present our first main result, which shows that the Eulerisable flows are flexible enough to encode any non-autonomous dynamics as above. Since these fields are stationary solutions of the Euler equations on some $(M,g)$, they exist for all time.

\begin{theorem}\label{T.main1}
The stationary Euler flows are universal. Moreover, the dimension of the ambient manifold $\mathbb S^n$ or $\mathbb R^n$ is the smallest odd integer~$n\in\{3\text{ dim }N+5,3\text{ dim }N+6\}$. In the time-periodic case, the extended field $u$ is a steady Euler flow with a metric $g=g_0+\delta_P$, where $g_0$ is the canonical metric on $\mathbb S^n$ and $\delta_P$ is supported in a ball that contains the invariant submanifold $e(N\times\mathbb S^1)$.
\end{theorem}
\begin{Remark}
The extension of the non-autonomous flow $u_0$ to an Eulerisable flow on, say, $\mathbb S^n$ is not unique. In fact, we prove that given any embedding $\tilde e:N\times \mathbb S^1 \rightarrow \mathbb S^n$, there exists a smooth embedding $e$ isotopic to $\tilde e$ and $C^0$-close to it which gives the Euler extension of $u_0$ introduced in Definition~\ref{def1}.
\end{Remark}

A striking corollary of this result, which illustrates the implications of the universality, is the embeddability of diffeomorphisms. We say that a (orientation-preserving) diffeomorphism $\phi:N\rightarrow N$ is \emph{Euler-embeddable} if there exists an Eulerisable field $u$ on $\mathbb S^n$ (for some $n$ that only depends on the dimension of $N$) with an invariant submanifold exhibiting a cross-section diffeomorphic to $N$ such that the first return map of $u$ at this cross section is conjugate to $\phi$.

\begin{Corollary}\label{Cor:difeo}
Let $N$ be a compact manifold and $\phi$ an orientation-preserving diffeomorphism on $N$. Then $\phi$ is Euler-embeddable in dimension $n$, where $n$ is the smallest odd integer $n\in\{3\text{ dim }N+5,3\text{ dim }N+6\}$.
\end{Corollary}

Let us mention a few words on the ideas of the proof of Theorem~\ref{T.main1}. The Eulerisable field $u$ that we construct on $\mathbb S^n$ (or $\mathbb R^n$) is nonvanishing and of Beltrami type with constant proportionality factor (notice that $n$ is an odd number). Using the correspondence between these fields and contact forms, the universality problem is then tantamount to studying the universality features of high-dimensional Reeb flows. A first difficulty is that the Reeb flows are geodesible, so their restriction to any invariant submanifold must be geodesible as well. Introducing the concept of \emph{Reeb embedding} of a compact manifold into a contact manifold, and using the flexibility (existence of an $h$-principle) of the isocontact embeddings, we prove that in fact geodesibility is the only obstruction for a vector field to be extendable to a Reeb flow on some contact manifold. A second difficulty is that the field $u_0$ that we want to extend is not generally geodesible, a problem that we address considering the suspension of the field.

A consequence of our methods of proof, which is of interest in itself, is an almost sharp novel embedding theorem for manifolds endowed with a geodesible flow into a contact manifold, so that the Reeb field of the ambient manifold for some contact form extends the geodesible field on the submanifold.

\begin{defi}\label{def:Remb}
{ Let $(N,X)$ be a geodesible field on a compact manifold. An embedding $e: (N,X) \rightarrow (M,\xi)$ of $N$ into a contact manifold $M$ is called a \emph{Reeb embedding} if there is a contact form $\alpha$ defining $\xi$ such that its Reeb vector field $R$ satisfies $e_*X=R$ (in particular $e(N)$ is an invariant submanifold of $R$). }
\end{defi}
{ The following result is {the main \emph{flexibility theorem} in this article:}}
\begin{theorem} \label{main2}
Let $e: (N,X) \rightarrow (M,\xi)$ be a embedding of $N$ into a contact manifold $(M,\xi)$ with $X$ a geodesible vector field on $N$. Then:
\begin{itemize}
\item If $\dim M \geq 3\dim N +2$, then $e$ is isotopic to a (small) Reeb embedding $\tilde e$, and $\tilde e$ can be taken $C^0$-close to $e$.
\item If $\dim M \geq 3\dim N$ and $M$ is overtwisted, then $e$ is isotopic to a Reeb embedding.
\end{itemize}
\end{theorem}
{ In view of the connection between Reeb and Beltrami fields which will be stated in Section~\ref{SS.Belt}, this theorem shows the flexible character of the steady Euler flows.} The notion of \emph{small Reeb embedding} in this statement will be introduced in Section~\ref{S4}. Moreover, we also obtain a full $h$-principle for what we call iso-Reeb embeddings (Reeb embeddings with certain fixed data) into overtwisted manifolds (Theorem~\ref{hp1}) and into general contact manifolds (Theorem~\ref{hp2}). We believe that these ideas may be useful to attack some purely geometric problems in Contact Topology.

Since Tao introduced the concept of universality to analyze the Turing completeness of the Euler equations~\cite{T0,T1}, we want to finish this introduction with an application of Theorem~\ref{T.main1} in this setting. We say that an Eulerisable flow on $\mathbb S^n$ is \emph{Turing complete} if the halting of any Turing machine with a given input is equivalent to a certain bounded trajectory of the flow entering a certain open set of $\mathbb S^n$ (what is known as the ``reachability problem'', see Section~\ref{SS5.2} for more details). This implies, in particular, that the flow has undecidable trajectories. Our second main result is the following.

\begin{theorem}\label{T.main2}
There exists an Eulerisable flow on $\mathbb S^{17}$ which is Turing complete.
\end{theorem}

The solution of the Euler equations that encodes a universal Turing machine provided by this theorem is stationary. We do not know if it gives rise to a global-time solution when it is considered as the initial condition for the Navier-Stokes equations on $\mathbb S^{17}$ with the corresponding Riemannian metric.

\begin{Remark}
In two sequels to this work \cite{CMPP2,JMPA}, we construct Turing complete steady Euler flows on a Riemannian three-dimensional sphere and in Euclidean $3$-space, respectively. These results are obtained by combining different techniques from contact geometry, symbolic dynamics and partial differential equations. Even taking into account that these constructions produce lower dimensional examples, the method of proof of Theorem \ref{T.main2} gives a machinery to construct Turing complete Eulerizable flows from initial Turing complete data in lower dimensions. This is done by converting the initial data into geodesible vector fields and using the fact that the  $h$-principle in the construction of Theorem \ref{main1} is algorithmic. In the proof of Theorem \ref{T.main2} the initial datum is a Turing complete diffeomorphism of the $4$-dimensional torus (see proposition \ref{Prop:Turing}) but it could be replaced by other Turing complete \emph{black boxes}.
\end{Remark}

\begin{Remark}
In this article we study the dynamics (integral curves) of the velocity field of a stationary solution of the Euler equations, which physically can be understood as the labels-to-particles map from Lagrangian coordinates to Eulerian coordinates. This map certainly has relevance in understanding the dynamics of the (time-dependent) Euler equations; for instance, the Lyapunov exponents of a steady Euler flow on $\mathbb T^3$ (flat metric) are related to the linear instability of these flows~\cite{FV93}. Since generically a Reeb field exhibits hyperbolic periodic orbits, one expects that the steady solutions we construct are linearly unstable for the evolution of the Euler equations in Sobolev spaces (a proof, however, would involve an extension of the theory in~\cite{FV93} to arbitrary Riemannian manifolds, which is not available). { The universality properties of the time-dependent Euler equations are investigated by Tao \cite{T2, T3}, and more recently by Torres de Lizaur \cite{TdL}.} In a subsequent work~\cite{CMP2} we answered the initial question posed by Tao~\cite{T2} by constructing time-dependent Turing complete solutions to the Euler equations in a manifold of high dimension.
\end{Remark}

The paper is organized as follows. In Section~\ref{S2} we review some classical results on contact geometry and $h$-principle, Beltrami flows and geodesible fields that will be instrumental in the next sections. In Section~\ref{S3} we study the extendibility properties of the geodesible fields to Reeb fields for some contact manifold and state several Reeb embedding theorems that will be used in the proof of the theorems above. For the benefit of the reader, the proof of the most technically demanding Reeb embedding theorem, which gives an ``almost optimal'' dimension for the ambient manifold, is postponed to Section~\ref{S4}. In Section~\ref{S5} we apply the previous results to the Euler equations to prove Theorems~\ref{T.main1} and~\ref{T.main2}, and Corollary~\ref{Cor:difeo}. Combining the results in Section~\ref{S3} with a number of $h$-principles for embeddings into contact manifolds~\cite{hprinc,BEM}, in Section~\ref{S4} we establish a fairly general $h$-principle for iso-Reeb embeddings. Finally, in Section~\ref{S:end} we provide some examples and generalizations of the iso-Reeb embedding theorems proved in Section~\ref{S3}, in particular depicting the space of iso-Reeb embeddings. Unless otherwise stated, all the manifolds and submanifolds in this paper are orientable, connected and have no boundary.

\section{Preliminaries}\label{S2}

In this section we review some concepts and results that will be instrumental in the forthcoming sections. In Subsection~\ref{SS.cont}, we recall the definition of a contact manifold and state some classical flexibility theorems for isocontact embeddings. The definition of a Beltrami field and its connection with the Reeb flows of a contact form is presented in Subsection~\ref{SS.Belt}. Finally, in Subsection~\ref{SS.geod} some basic facts about geodesible vector fields are introduced.

\subsection{Contact geometry and $h$-principle for isocontact embeddings}\label{SS.cont}


Let $M^{2m+1}$ be an odd dimensional manifold equipped with a hyperplane distribution $\xi$. Assume that there is a $1$-form $\alpha \in \Omega^1(M)$ with $\ker \alpha = \xi$ and $\alpha \wedge (d\alpha)^m >0$ everywhere. Then we say that $(M^{2m+1},\xi)$ is a (cooriented) contact manifold.

The $1$-form $\alpha$ is called a contact form. Of course, the contact structure $\xi$ does not depend on the choice of the defining 1-form $\alpha$. It is well known that $d\alpha$ induces a symplectic structure on the hyperplane distribution $\xi$ (of even dimension $2m$). The unique Reeb vector field $R$ associated to a given contact form $\alpha$ is defined by the equations
\begin{equation} \label{eq:Reeb}
\iota_R\alpha=1\,, \qquad \iota_Rd\alpha =0\,.
\end{equation}


The world of contact geometry exhibits a lot of flexibility which usually enables to use arguments from differential topology to prove geometric properties. The pioneering work of Gromov~\cite{G} shows that there exists a parametric $h$-principle for contact structures on open manifolds. For general manifolds, a parametric and relative $h$-principle was proved in~\cite{BEM} using overtwisted disks, see also~\cite{El} and~\cite{CPP} for previous results.

Grosso modo, the general philosophy of the $h$-principle leans on the idea of deforming \emph{formal} solutions into \emph{honest solutions} of an equation (PDE or, more generally, a partial differential relation). When this is possible, finding a solution is simplified to a homotopic-theoretical problem. A reincarnation of this principle in the contact set-up requires a fine inspection of the notion of formal contact structure. Specifically, the topological information given by the contact distribution consists of the codimension one distribution $\xi$ and the symplectic structure on it induced by $d\alpha$. In fact, only the conformal class is determined because a rescaling $\alpha'=f\alpha$ is a contact form for the same contact structure. This allows one to introduce the concept of a formal contact structure that is defined as a cooriented hyperplane distribution and a conformally symplectic class on it. In the literature this structure has been usually called almost contact structure, however in the last few years the term formal has become standard since it implements the formal condition for the $h$-principle. We can find a $2$-form $\omega$ on $M$ such that $(\xi,\omega|_\xi)$ is a conformal symplectic vector bundle: we say that two formal contact structures defined as $(\xi , \omega|_\xi)$ and $(\xi , \hat\omega|_\xi)$ are equivalent if $\omega$ and $\hat\omega$ are conformally equivalent. So a formal contact structure is described by a codimension one conformal symplectic vector bundle.

The flexibility statements that we need in this paper concern isocontact embeddings. Recall that a map $f: (N,\xi_N)\rightarrow (M,\xi_M)$ between contact manifolds is called isocontact if $f_*\xi_N=\xi_M$. In the formal level, a monomorphism $F:TN\rightarrow TM$ is called isocontact if $\xi_N= F^{-1}(\xi_M)$ and $F$ induces a conformally symplectic map with respect to the conformal symplectic structures $CS(\xi_N)$ and $CS(\xi_M)$. The following $h$-principle was proved in~\cite[Section 12.3.1]{hprinc}. We recall that $N_0$ is called a core of an open manifold $N$ if for an arbitrarily small neighborhood $U$ of $N_0$, there is an isotopy which sends diffeomorphically $U$ to $N$.

\begin{theorem}[\cite{G}]\label{isoc}
Let $(N,\xi_N)$ and $(M,\xi_M)$ be contact manifolds of dimension $2n+1$ and $2m+1$ respectively. Let $f_0=:(N,\xi_N) \rightarrow (M,\xi_M)$ be an embedding such that its differential $F_0:=df_0$ is homotopic (via monomorphisms $F_t: TN \rightarrow TM$, with projections onto the base given by $f_0$) to a conformal symplectic monomorphism $F_1$. Then
\begin{itemize}
\item If $N$ is open and $n\leq m-1$ then there is an isotopy $f_t: N\rightarrow M$ such that the embedding $f_1$ is isocontact and $df_1$ is homotopic to $F_1$ through conformal symplectic monomorphisms. Given a core $N_0$ of $N$, $f_t$ can be chosen arbitrarily $C^0$-close to $f_0$ near $N_0$.
\item If $N$ is closed and $n \leq m-2$ then the above $f_t$ exists. Moreover, one can choose $f_t$ to be arbitrarily $C^0$-close to $f_0$.
\end{itemize}
\end{theorem}

In~\cite{BEM}, the authors showed that every formal contact structure is deformable to a genuine contact structure, thus proving the long standing conjecture of the existence of contact structures in every formal contact manifold. Restricting to a particular class of contact structures called \emph{overtwisted}, a full $h$-principle was proved, thus implying a result stronger than Theorem~\ref{isoc} for isocontact embeddings into overtwisted manifolds. This result, which holds for codimension $0$ isocontact embeddings of open manifolds, can be summarized as follows:

\begin{theorem}[\cite{BEM}]\label{CorBEM}
Let $(M^{2m+1}, \xi)$ be a connected overtwisted contact manifold and $(N^{2m+1}, \zeta)$ an open contact manifold of the same dimension. Let $f: N \rightarrow M$ be a smooth embedding covered by an isocontact bundle homomorphism $\varphi: TN \rightarrow TM$, that is such that $\varphi(\zeta_x)=\xi_{f(x)}$ and $\varphi$ preserves the conformal symplectic structures on the distributions. If $df$ and $\varphi$ are homotopic as injective bundle homomorphisms then $f$ is isotopic to an isocontact embedding $\tilde f$.
\end{theorem}

\subsection{The correspondence between Beltrami fields and Reeb flows}\label{SS.Belt}

The stationary Euler equations~\eqref{euler_stat} can be written equivalently using differential forms as
\begin{equation}
\iota_ud\alpha=-dB\,, \qquad d\iota_u \mu=0\,,
\end{equation}
where $\alpha:=\iota_ug$ is the metric dual $1$-form of $u$ and $\mu$ is the Riemannian volume form. The function $B$ is called the Bernoulli function, and is defined as $B:=p+\frac12g(u,u)$. When the dimension of $M$ is odd, i.e. $\text{dim }M=2m+1$, one can introduce the concept of vorticity, which is a vector field that plays a fundamental role in fluid mechanics. It is defined as the curl of $u$, i.e. $\omega:=\curl u$, where the curl of a vector field in odd dimensions is computed as the only field that satisfies the equation
$$\iota_\omega \mu= (d\alpha)^m\,.$$
Notice that when $m>1$, the curl is a nonlinear differential operator.

A landmark in the study of the steady Euler flows in odd dimensional manifolds is Arnold's structure theorem~\cite{Ar65,AK}, which is based on the following dichotomy: a stationary solution either has a nontrivial first integral (the Bernoulli function) or it is a \emph{Beltrami field}. We recall that a Beltrami field $u$ is a vector field that satisfies
\begin{equation}
\curl u=fu\,, \qquad \Div u=0\,,
\end{equation}
for some smooth function $f$. These fields are stationary solutions of the Euler equations with constant Bernoulli function. { Even if in this case the Bernoulli function is a trivial first integral, a Beltrami field might as well have another first integral, i.e. both situations are not mutually exclusive.}

When the function $f$ does not vanish, say $f>0$, there is a remarkable correspondence between nonvanishing Beltrami fields and Reeb flows for a certain contact structure suggested by  Sullivan and proved by Etnyre and Ghrist in~\cite{EG} (see also~\cite{CMP0} for an extension to Beltrami fields on manifold with cylindrical ends and $b$-contact structures). This result paves the way to study the stationary Euler equations using contact geometry techniques. The proof presented in~\cite{EG} is in dimension $3$, but it can be readily extended to any higher odd dimension. We include it here for the sake of completeness.

\begin{theorem}\label{correspondence}
Any nonvanishing Beltrami field with positive proportionality factor is a reparametrization of a Reeb flow for some contact form. Conversely, any reparametrization of a Reeb vector field of a contact structure is a nonvanishing Beltrami field for some Riemannian metric.
\end{theorem}

\begin{proof}
Let $u$ be a Beltrami field and $\alpha$ its metric-dual 1-form. As $\dim M=2m+1$, the Beltrami equation reads as
$$ (d\alpha)^m = f \iota_u \mu\,. $$
Since $f>0$ and $u$ does not vanish, it then follows that $\alpha$ satisfies the contact condition
$$ \alpha \wedge (d\alpha)^m= f \alpha \wedge \iota_u \mu > 0.$$
Moreover, $\iota_u (d\alpha)^m=0$ so $u \in \ker d\alpha$. Accordingly, $u$ is a reparametrization of the Reeb field $R$, that is $R=\frac{u}{g(u,u)}$.

To prove the converse implication, consider a contact form $\alpha$ and its associated Reeb flow $R$. Let $u:=R/h$ where $h>0$ is a reparametrization of $R$, and take an almost-complex structure $J$ on $\ker \alpha= \xi$ adapted to $d\alpha$, i.e. $d\alpha(\cdot, J\cdot)$ is a positive definite quadratic form on $\xi$. Define the metric
\begin{equation}\label{met}
g(u,v) := h(\alpha(u)\otimes \alpha(v)) + \tilde h \,d\alpha(u,Jv)\,,
\end{equation}
where $\tilde h$ is any positive function. It then follows that $\iota_u g= \alpha$. It is clear that the function $\tilde h$ can be chosen so that the Riemannian volume form is $\mu= h\alpha\wedge (d\alpha)^m$. Therefore, $u$ is a Beltrami field (with factor $f=1$) with respect to the metric $g$ because $\iota_u \mu= (d\alpha)^m$ and $d\iota_u \mu =0$.
\end{proof}
\begin{Remark}
The (non-unique) Riemannian metric~\eqref{met} is called an \emph{adapted metric} to the contact form $\alpha$.
\end{Remark}

\subsection{Geodesible vector fields}\label{SS.geod}

We say that a vector field $X$ on a manifold $M$ is geodesible if there exists a Riemannian metric $g$ on $M$ making the orbits geodesics. It can be shown~\cite{Gl,S} that the geodesibility condition is equivalent to the existence of a $1$-form $\beta$ such that $\beta(X)>0$ and $\iota_Xd\beta=0$. If we also assume that the $1$-form can be taken so that $\beta(X)=1$, we say that $X$ is of unit length. Unless otherwise stated, all along this paper we shall assume that a geodesible field has unit length. A similar characterization for Eulerisable flows was introduced in~\cite{P1}.

Another characterization that we shall use later is that $X$ is geodesible of unit length if and only if it preserves a transverse hyperplane distribution. The necessity is immediate from the aforementioned result. To prove that it is sufficient, let $\eta$ be the hyperplane distribution and $\beta$ a defining $1$-form such that $\ker \beta=\eta$ and $\beta(X)>0$. Dividing $\beta$ by the function $\beta(X)$ we can safely assume that $\beta(X)=1$. The condition that $X$ preserves $\ker \beta$ is tantamount to saying that
$$\mathcal{L}_X\beta = f\beta\,,$$
for some function $f\in C^\infty (M)$. Cartan's formula implies that $\iota_Xd\beta=f\beta$, and contracting with the vector field $X$ we finally conclude that $f=0$.

In the next sections we shall usually denote a geodesible field by $(N,X)$ or $(N,X,\eta)$, where $N$ is the ambient manifold, $X$ is the field and $\eta$ is the transverse hyperplane distribution preserved by~$X$. In particular we might fix the $1$-form $\beta$, and hence the hyperplane distribution $\eta=\ker \beta$ preserved by $X$.

\section{Reeb-embeddability and geodesible fields}\label{S3}

The goal of this section is to prove the following theorem:

\begin{theorem} \label{main1}
Let $(N,X)$ be a compact manifold endowed with a geodesible field $X$. Then there is a smooth embedding $e:N\rightarrow \mathbb S^{n}$ with $n=4\,\text{dim }N-1$ and a $1$-form $\alpha$ defining the standard contact structure $\xi_{std}$ on $\mathbb S^{n}$ such that $e(N)$ is an invariant submanifold of the Reeb field $R$ defined by $\alpha$ and $e_*X=R$. Moreover, $\alpha$ is equal to the standard contact form $\alpha_{std}$ in the complement of a ball that contains $e(N)$.
\end{theorem}

To prove this result, we first recall (Subsection~\ref{Inaba}) Inaba's characterization of the vector fields on a submanifold of a contact manifold $(M,\xi)$ that can be extended as Reeb flows for some contact form defining the contact structure $\xi$. In Subsection~\ref{S:RG} we introduce the concept of \emph{Reeb embedding} and prove Theorem~\ref{main1} using an $h$-principle for isocontact embeddings. Finally, in Subsection~\ref{SS.another} we state a stronger Reeb embedding result (Theorem~\ref{main2}) which substantially improves the dimension $n$ in Theorem~\ref{main1} and shows that, roughly speaking, any embedding of high enough codimension can be deformed into a Reeb embedding. The proof of this result is more involved and will be postponed to Section~\ref{S4}.

We shall see in Section~\ref{S5} how these results can be used, in combination with the correspondence theorem in Subsection~\ref{SS.Belt}, to prove the universality results stated in Section~\ref{S1}. As an immediate corollary we obtain:

\begin{Corollary}\label{Beltr}
Let $(N,X)$ be a compact manifold endowed with a geodesible field $X$ which is not necessarily of unit length. Then there exists an embedding $e:N\rightarrow \mathbb S^{n}$ with $n=4\,\text{dim }N-1$ and a non-vanishing Beltrami field $u$ on $\mathbb S^{n}$ with constant proportionality factor such that $e_*X=u$. The Riemannian metric for which $u$ is a Beltrami field is the canonical metric of $\mathbb S^n$ in the complement of a ball containing $e(N)$.
\end{Corollary}
\begin{proof}
Reparametrizing $X$ we obtain another geodesible vector field $\tilde X$ of unit length. Theorem~\ref{main1} implies that $(N,\tilde X)$ admits an embedding into $(\mathbb S^{n},\xi_{std})$, $n=4\,\text{dim }N-1$, such that there is a defining contact form $\alpha$ whose Reeb vector field satisfies $R|_{e(N)}=\tilde X$. Obviously, we can now reparametrize $R$ by a function $f$ such that $fR|_{e(N)}=X$ and $f=1$ in the complement of a ball $B$ that contains $e(N)$. By Theorem~\ref{correspondence}, the vector field $fR$, which is no longer Reeb in general, is a Beltrami field with constant proportionality factor for some Riemannian metric on $\mathbb S^n$ that can be taken to be the canonical metric in the complement of $B$.
\end{proof}

\subsection{Extension of Reeb flows}\label{Inaba}

We recall a simple characterization due to Inaba~\cite{I} of the vector fields on a submanifold of a contact manifold that can be extended to a Reeb vector field. For the sake of completeness, we include a concise proof.

\begin{lemma}\label{L.Inaba}
Let $(M,\xi)$ be a (cooriented) contact manifold and $(N,X)$ a compact submanifold of $M$ endowed with a tangent (non-vanishing) vector field $X$ which is positively transverse to $\xi$ on $N$. Then there is a contact form $\alpha$ defining $\xi$ such that its Reeb vector field $R$ satisfies $R|_N=X$ if and only if $X$ preserves $TN \cap \xi$.
\end{lemma}
\begin{proof}
The necessity is trivial because a Reeb vector field preserves the contact distribution. To prove the sufficiency, assume that the vector field $X$ on $N$ preserves the tangent distribution $\eta:=TN\cap \xi$. It is useful to denote the embedding of $N$ into $M$ by $e:N\rightarrow M$, where with a slight abuse of notation we are identifying $N$ with its embedded image.


Let $\alpha_0$ be a defining contact form of $\xi$. Fix the strictly positive smooth function $h_N$ on $N$, given as $h_N:=\frac{1}{e^* \alpha_0(X)}$. By using partitions of unity, we can find a strictly positive function $h: M \to \R^+$ such that $h|_N=h_N$.
Define a new $1$-form $\alpha_1:=h\alpha_0$, still associated to the contact structure $\xi$, which by construction satisfies the first condition in the defining Reeb equations (\ref{eq:Reeb})
$$ \iota_X \alpha_1 =1\,.$$
Since $X$ preserves $\ker \alpha_1$, it preserves $\ker e^*\alpha_1$. Hence this reads as,
\begin{equation*}
\mathcal{L}_Xe^*\alpha_1=fe^*\alpha_1,
\end{equation*}
where $f$ is a smooth function. By Cartan formula this implies that $\iota_Xe^*d\alpha_1=fe^*\alpha_1$. Contracting this equation (in $1$-forms) with the vector field $X$, we immediately obtain that $f=0$. Thus, we have
\begin{equation} \label{eq:nice}
\iota_X d e^*\alpha_1 = 0.
\end{equation}
Now we want to find a new associated contact form multiplying by a strictly positive smooth function $\lambda$ on $M$ such that the vector field $X$ satisfies the Reeb equations (\ref{eq:Reeb}) when applied to the $1$-form $\alpha:=\lambda\alpha_1$. Taking a function $\lambda$ such that $\lambda|_N=1$, by the uniqueness of the Reeb vector field, this is tantamount to saying that $X$  verfies $\iota_Xd(\lambda\alpha_1)= 0$, since this new contact form still satisfies $e^*(\alpha) (X)=1$. Thus, we just need to find a function $\lambda$ such that the $1$-form $\lambda \alpha_1$ satisfies the second Reeb equation in (\ref{eq:Reeb}), i.e.
$$  \iota_Xd(\lambda \alpha_1)=0\,,$$
on $TM|_N$. Expanding it we obtain
\begin{align*}
	0=\iota_X d(\lambda\alpha_1)&= \iota_X (d\lambda \wedge \alpha_1 +  \lambda d\alpha_1)\,,
\end{align*}
that restricted to $N$ reads as
\begin{align} \label{eq:trivial}
	0= -d\lambda + d\alpha_1(X)\,,
\end{align}
where we have used that $X$ is tangent to $N$ and $\lambda|_N=1$. Accordingly, the condition that $\lambda$ must satisfy reads as $d\lambda=d\alpha_1(X)$ on $N$ (over the whole $TM|_{N}$).

Since we proved above that $\iota_Xe^*d\alpha_1=0$, and $\lambda|_N=1$, the equality of $1$-forms (\ref{eq:trivial}) holds on $TN\subset TM$. For the normal directions, just find a smooth function such that the partial derivatives for any $v \in T_pM$ with $p\in N$ satisfy $\frac{\partial \lambda}{\partial v}= d\alpha_1(X,v)$. This determines the whole $1$-jet of the function $\lambda$ on $N$. Again, by a standard argument taking partitions of the unity, this implies the existence of a positive smooth function $\lambda$ on $M$ that extends this given $1$-jet. The lemma then follows.
\end{proof}

\begin{Remark}
We remark that the vector field $X$ in Lemma~\ref{L.Inaba} is geodesible. Indeed, following the notation of the proof of the lemma, the $1$-form $\beta:=e^*\alpha$ satisfies that $\iota_X\beta=1$ and $\iota_Xd\beta=0$, which implies the geodesibility according to Subsection~\ref{SS.geod}.
\end{Remark}

\begin{Remark}\label{rem:complementstd}
It follows from the proof of the lemma, that if $\alpha_0$ is an associated contact form for $\xi$, then the $1$-form $\alpha$ can be taken to be equal to $\alpha_0$ in the complement of a neighborhood of $N\subset M$ (just take extensions of the functions $h$ and $\lambda$ in the proof so that $h=\lambda=1$ in the complement of the neighborhood).
\end{Remark}

\subsection{Existence of Reeb embeddings}\label{S:RG}

The characterization of vector fields extendible to Reeb flows presented in the previous subsection suggests the following definition { (see also Definition \ref{def:Remb})}

\begin{defi}\label{d:reeb}
Let $(N,X)$ be a geodesible field on a compact manifold. An embedding $e: (N,X) \rightarrow (M,\xi)$ of $N$ into a contact manifold $M$ is called a \emph{Reeb embedding} if there is a contact form $\alpha$ defining $\xi$ such that its Reeb vector field $R$ satisfies $e_*X=R$ (in particular $e(N)$ is an invariant submanifold of $R$). If we further assume that the geodesible vector field comes with a fixed preserved distribution $\ker \beta=\eta$, then an embedding is called an \emph{iso-Reeb embedding} if $e^*\xi=\eta$.
\end{defi}

Observe that a Reeb embedding $e:(N,X)\rightarrow (M,\xi)$ clearly induces an iso-Reeb embedding just by declaring $\eta:=e^*\xi$.
As noticed before, any Reeb vector field tangent to a submanifold is geodesible on it. Theorem~\ref{main1} then claims that the converse also holds, i.e. that for any geodesible field $(N,X)$ there exists a Reeb embedding into a high-dimensional sphere endowed with the standard contact structure. The following technical lemma is key to prove the main result of this section. For the proof, we follow~\cite[Section 16.2.2]{hprinc}.

\begin{lemma}\label{tech}
Let $(N,X,\eta)$ be a geodesible field on a compact manifold $N$ of dimension $n_0$, and $\beta$ a defining $1$-form of the hyperplane distribution $\eta$. Assume that there exists an embedding $e:(N,X,\eta)\rightarrow M$ into a manifold of dimension $2m-1$ endowed with a hyperplane distribution $\xi$ defined on $e(N)$ such that
\begin{enumerate}
\item $\eta=e^* \xi$.
\item There is a nondegenerate $2$-form $\omega$ on $\xi|_N$.
\item $\omega|_{TN}=d\beta$.
\end{enumerate}
Then there is a small neighborhood $U$ of $e(N)$ in $M$ and a contact form $\alpha$ on $U$ such that $e^*\alpha=\beta$.
\end{lemma}
\begin{proof}
For notational simplicity, we shall identify $N$ with its embedded image $e(N)$. Consider a small neighborhood $U\subset M$ of $N$, which can be identified with a normal disk bundle $\pi:U\rightarrow N$. Fix a covering by small open sets $V_j\subset N$ where $U|_{V_j}=V_j\times\mathbb R^{m'-1}$, $m':=2m-n_0$. Since $e^*\xi = \eta$, then the hyperplane distribution $\xi$ on $N$ can be split as $\xi|_{V_j}=\eta\oplus \mathbb R^{m'-1}$. In terms of this splitting, the assumption $\omega|_{TN}=d\beta$ implies that the $2$-form $\omega$ can be written as
$$
\omega=d(\pi^*\beta)+\omega'
$$
where $\omega'$ is a $2$-form that satisfies that $\omega'|_{\eta}=0$.

Let us introduce coordinates $(y_{1,j},\cdots,y_{m'-1,j})$ in the second factor of $V_j \times \mathbb R^{m'-1}$. In these coordinates, we can assume that $\omega'$ has the form
$$\omega'= \sum_{k=1}^{ m'-1} dy_{k,j} \wedge \beta_{k,j}$$
where $\beta_{k,j}$ are suitable $1$-forms on $N$. Now we can define on $U$ the $1$-form
$$\alpha_j:= \pi^*\beta + \sum_{k=1}^{m'-1}{y_{k,j} \beta_{k,j}}\,.$$
Notice that $e^*\alpha_j=\beta$ and that $\ker \alpha_j|_N=\xi$. Moreover, since
\begin{align*}
(d\alpha_j)|_N&=d(\pi^*\beta)+\sum_{k=1}^{m'-1}(dy_{k,j}\wedge \beta_{k,j}|_N + y_{k,j}\wedge d\beta_{k,j}|_N) \\
&=  d(\pi^*\beta)+\sum_{k=1}^{m'-1}dy_{k,j}\wedge \beta_{k,j}|_N  \\
&=\omega|_N,
\end{align*}
is nondegenerate on $\xi|_N$, and $X$ is transverse to $\xi$. Observe that $\alpha_j$ is a contact form { near the zero section of} $V_j\times \R^{m'-1}$. Choose a partition of the unity $\chi_j$ subordinated to the open cover $V_j$, and define $\alpha=\sum_j \chi_j \alpha_j$.
Expanding this expression we get that:
\begin{align*}
(d\alpha)|_N&=  d(\pi^*\beta)+\sum_j \chi_j \sum_{k=1}^{m'-1}dy_{k,j}\wedge \beta_{k,j}|_N  \\
&=\omega|_N.
\end{align*}
This concludes the proof.
\end{proof}

\begin{proof}[Proof of Theorem~\ref{main1} :]
Following the notation introduced in Lemma~\ref{tech}, $\eta$ is the hyperplane distribution on $N$ preserved by $X$, and $\beta$ is a defining $1$-form. Consider the vector bundle $M$ defined by the dual distribution $\eta^*$ over $N$, and denote the bundle projection as $\pi:M \rightarrow N$. Observe that $\text{dim }M=2n_0-1$ and that the tangent bundle of $M$ at the zero section (which is the manifold $N$) splits as $TM|_N=TN\oplus\eta^*=\langle X \rangle \oplus \eta\oplus \eta^*$.
This distribution $\eta \oplus \eta^*$ over $N$ has dimension $2n_0-2$ and is equipped with the canonical symplectic form $\omega_0$ defined by
$$\omega_0((v_1\oplus \alpha_1),(v_2\oplus \alpha_2)):= \alpha_2(v_1) - \alpha_1(v_2)\,,$$
where $v_k$ is a section of $\eta$ and $\alpha_k$ is a section of $\eta^*$. Observe that with this symplectic structure $\eta$ on $N$ is an isotropic subspace, i.e. denoting by $j:\eta\rightarrow \eta\oplus \eta^*$ the natural inclusion, we have that $j^*\omega_0=0$.

Let us now perturb the symplectic structure $\omega_0$ by lifting a $2$--form on $N$ to $TM|_N$. For every point $p \in N$ we define the $2$-form
$$\omega_N|_p= A(\omega_0)|_p + (d\beta)|_p, $$
where $A>0$ is a constant large enough so that $\omega_N$ defined on $N$ is still nondegenerate.
It follows from the previous construction that, if $e_0:(N,X,\eta)\rightarrow M$ denotes the natural inclusion then we can apply Lemma~\ref{tech} to conclude that there is a contact form $\alpha$ in a neighborhood $U$ of $N$ in $M$ such that $e_0^*\alpha=\beta$. Notice that the contact distribution $\xi:=\ker \alpha$ coincides with $\eta \oplus \eta^*$ on $N$.

Summarizing, we have constructed an open contact manifold $U$ of dimension $2n_0-1$ with a submanifold $N$ endowed with a vector field $X$ that is positively transverse to the contact distribution $\xi$ and preserves $TN\cap \xi=\eta$. { Lemma~\ref{lem:sph} below and the $h$-principle for isocontact embeddings Theorem \ref{isoc}} imply that $U$ can be {isocontactly} embedded into $(\mathbb S^{n},\xi_{std})$ for $n=4n_0-1$. Denoting this embedding by $e:U\rightarrow \mathbb S^{n}$ it obviously satisfies $e_0^*e^*\xi_{std}=\eta$. Identifying $N$ with its embedded image in $\mathbb S^{n}$ (via the embedding $e\circ e_0$), we have that the field $X$ preserves $TN \cap \xi_{std}= \eta$, so we can apply Lemma~\ref{L.Inaba} to conclude that there is a contact form $\tilde\alpha$ whose Reeb field $R$ coincides with $X$ on $N$, and $\tilde\alpha=\alpha_{std}$ in the complement of a neighborhood of $N$. The theorem then follows.
\end{proof}

\begin{lemma}\label{lem:sph}
Any smooth embedding of a contact manifold $(N^{2n_0-1},\eta)$ into $(\mathbb S^{4n_0-1},\xi_{std})$ is a formal isocontact embedding.
\end{lemma}
\begin{proof}
Let $f:N^{2n_0-1}\rightarrow \mathbb S^{4n_0-1}$ be a smooth embedding. Let us construct a family of monomorphisms  $F_t: TN^{2n_0-1} \rightarrow T\mathbb S^{4n_0-1}|_N $ such that $F_0=df$, $F_1(R_{\eta})=R_{std}$ (the corresponding Reeb fields), $F_1(\eta)\subset \xi_{std}$ and $F_1$ is a complex monomorphism. To this end, we first find a family of vector fields $R_t$ over $T\mathbb S^{4n_0-1}|_N$ such that $R_0=R_{\eta}$ and $R_1=  R_{std}$. This family exists because the connectedness of the sphere is higher than the dimension of $N$, i.e., any two sections of $T\mathbb S^{4n_0-1}|_N$ are homotopic through non-vanishing sections since $4n_0> 2n_0-1$. Now fix $\xi_t$ to be any complementary of $R_t$ which satisfies $f_* \eta \subset \xi_1$. This automatically provides a family, canonical up to homotopy of (real) monomorphisms, $F_t: TN \to T\mathbb S^{4n_0-1}|_N$ such that $F_0=df$. It remains to show that $F_1: \eta \to \xi_{std}|_N$ is homotopic { (by an homotopy fixed along the base)} to a complex monomorphism, but this is an easy consequence of the connectedness of the inclusion map of the space of complex monomorphisms into the space of real monomorphisms, i.e., the rank of connectedness is bigger than the dimension of $N$.
\end{proof}

\begin{Corollary}
Let $X$ be a nonvanishing vector field on a compact manifold $N$. Then $N$ embeds into some contact manifold $(M,\xi)$ such that $X=R|_N$ for a Reeb vector field $R$ of some contact form if and only if $X$ is geodesible.
\end{Corollary}

\begin{Remark}
The isocontact embedding theorem used in the proof of Theorem~\ref{main1} works for any ambient contact manifold of dimension $n=4\,\text{dim }N-1$ (because it gives an embedding into a Darboux neighborhood of any contact manifold of dimension bigger or equal than $4\,\text{dim }N-1$). This implies, in particular, that the ambient manifold in Theorem~\ref{main1} can be taken to be $(\mathbb R^{n},\xi_{std})$.
\end{Remark}

\begin{Remark}\label{open}
When the manifold $N$ is non-compact the following observation allows one to prove a result analogous to Theorem~\ref{main1}. Indeed, Lemma~\ref{L.Inaba} works if $N$ is a properly embedded submanifold, and the embedding provided by Whitney embedding theorem can be taken proper~\cite{Lee}. Accordingly, Theorem~\ref{main1} provides a Reeb embedding of any pair $(N,X)$ with $N$ non-compact and $X$ geodesible into $(\R^{n},\xi_{std})$, $n=4\,\text{dim }N-1$.
\end{Remark}

\subsection{An improved Reeb embedding theorem}\label{SS.another}
 Theorem~\ref{main1} shows the existence of a Reeb embedding of $(N,X)$ into $\mathbb S^n$ for $n=4\,\text{dim }N-1$. This suggests two questions:
\begin{enumerate}
\item Can we improve the bound on the dimension $n$ of the target space?
\item Can an embedding $e:(N,X) \rightarrow (M,\xi)$ be deformed into a Reeb embedding via an isotopy which is $C^0$-close to the identity?
\end{enumerate}

Let us finish this section by stating Theorem~\ref{main2}, which is a generalization of Theorem~\ref{main1}, and answers these questions. Its proof, which makes use of some non trivial modern $h$--principle results in contact topology, is technically much more involved than the proof of Theorem~\ref{main1}, and will be presented in Section~\ref{S4} together with a few corollaries that can be useful for other applications in Contact Geometry. This theorem is key for the proofs of Theorems~\ref{T.main1} and~\ref{T.main2} stated in the Introduction.

\begin{theorem*}[Theorem~\ref{main2}]
Let $e: (N,X) \rightarrow (M,\xi)$ be a embedding of $N$ into a contact manifold $(M,\xi)$ with $X$ a geodesible vector field on $N$. Then:
\begin{itemize}
\item If $\dim M \geq 3\dim N +2$, then $e$ is isotopic to a (small) Reeb embedding $\tilde e$, and $\tilde e$ can be taken $C^0$-close to $e$.
\item If $\dim M \geq 3\dim N$ and $M$ is overtwisted, then $e$ is isotopic to a Reeb embedding.
\end{itemize}
\end{theorem*}

The notions in the statement will be introduced in Section~\ref{S4}. For the proofs of Theorems~\ref{T.main1} and~\ref{T.main2} the (weaker) statement that provides a general Reeb embedding is sufficient. Roughly speaking, Theorem~\ref{main2} shows that Reeb embeddings are completely determined by differential topology invariants. This fact can be easily encoded in the $h$-principle philosophy (see~\cite{hprinc,G}), details will be provided in Section~\ref{S4}. As a Corollary we obtain the following improved version of Corollary~\ref{Beltr}; the proof is analogous so we omit it.
\begin{Corollary}\label{Cor:Belt}
Let $(N,X)$ be a compact manifold endowed with a geodesible field $X$ which is not necessarily of unit length. Then there exists an embedding $e:N\rightarrow \mathbb S^{n}$ with $n$ the smallest odd integer $n\in\{3\,\text{dim }N+2,3\,\text{dim }N+3\}$, and a non-vanishing Beltrami field $u$ on $\mathbb S^{n}$ with constant proportionality factor such that $u|_{e(N)}=X$. The Riemannian metric for which $u$ is a Beltrami field is the canonical metric of $\mathbb S^n$ in the complement of a ball containing $e(N)$.
\end{Corollary}

\section{Steady Euler flows: proof of Theorems~\ref{T.main1} and~\ref{T.main2}}\label{S5}

Our goal in this section is to apply the results on Reeb embeddings  in Section~\ref{S3} to prove the main theorems stated in the Introduction on the universality of the stationary Euler equations.

\subsection{Non-autonomous dynamics and universality}\label{SS5}

Let $u_0(\cdot,t)$ be a non-autonomous vector field on a compact manifold $N$, and assume that it is $2\pi$ periodic in $t$. The suspension of $u_0$ on the manifold $N\times \mathbb S^1$ (with $\mathbb S^1=\mathbb R/(2\pi\mathbb Z)$) is another vector field defined as
$$ X(x,\theta):=  u_0(x,\theta)+\partial_\theta\,,$$
with $x\in N$ and $\theta\in \mathbb S^1$.

The vector field $X$ on $N\times \mathbb S^1$ is geodesible. Indeed, the closed $1$-form $\beta:= d\theta$ obviously satisfies that $\beta(X)= 1$ and $\iota_Xd\beta=0$, which implies that $X$ is geodesible, c.f. Subsection~\ref{SS.geod}. Now, applying Theorem~\ref{main2} to the pair $(N\times \mathbb S^1,X)$, we conclude that there exists a Reeb embedding $e:(N\times \mathbb S^1,X)\rightarrow (\mathbb S^n,\xi_{std})$ for the smallest odd integer $n\in \{3\dim\, N+5,3\dim\,N+6\}$. This means, c.f. Definition~\ref{d:reeb}, that there is a defining $1$-form $\alpha$ of $\xi_{std}$ whose Reeb field $R$ satisfies that $R|_{e(N\times \mathbb S^1)}=X$. { {Furthermore}, we can require that $\alpha=\alpha_{std}$ in the complement of a ball $B$ that contains $e(N\times \mathbb S^1)$, as per Remark \ref{rem:complementstd}.}

It follows from the Beltrami-Reeb correspondence Theorem~\ref{correspondence}, that $R$ is a Beltrami field (and hence a steady Euler flow) for some metric $g$ on $\mathbb S^n$. Moreover, since the adapted metric to the standard contact form on the sphere is the round metric $g_0$, it turns out that $g=g_0$ in the complement of $B\subset \mathbb S^n$.

Setting $u:=R$, the previous construction shows that any (time-periodic) non-autonomous dynamics $u_0$ is Euler-extendible, recall Definition~\ref{def1}.

The general case of a non-autonomous flow $u_0(\cdot,t)$ is analogous. The suspension manifold is $N\times \mathbb R$ and $X$ is defined as above with $\theta\in\R$. This vector field is geodesible, so proceeding as before we conclude that $u_0$ is Euler-extendible to $\R^n$, for the smallest odd integer $n\in\{3\dim\,N+5,3\dim\,N+6\}$. Note that in this case the adapted metric to the standard contact form on $\R^n$ is not the Euclidean one. This completes the proof of Theorem~\ref{T.main1}.

\begin{Remark}
When the extended manifold is $\mathbb S^n$, the steady Euler flow $u$ is equal to the Hopf field in the complement of $B$ (because the Hopf field is the Reeb field associated to the standard contact form). In the case that the extension is in $\R^n$, the vector field $u$ is the vertical field $\partial_{x_n}$ in the complement of a neighborhood of the non-compact manifold $e(N\times \R)$.
\end{Remark}

\begin{Remark}
When the vector field $u_0$ is autonomous and geodesible (not necessarily of unit length) we do not need to take the suspension of $u_0$. In this case we can directly apply Corollary~\ref{Cor:Belt} to conclude that $(N,u_0)$ can be embedded into $\mathbb S^n$, where $n$ is the smallest odd integer $n\in\{3\,\dim N+2,3\,\dim\,N+3\}$, so that $e_*u_0$ extends as a Beltrami field with constant proportionality factor on $\mathbb S^n$.
\end{Remark}

We conclude this subsection by proving Corollary~\ref{Cor:difeo}. The main idea is again a suspension construction, depicted in Figure \ref{sus}.

\begin{figure}[!h]
\begin{tikzpicture}
\node[anchor=south west,inner sep=0] at (0,0) {\includegraphics[scale=0.115]{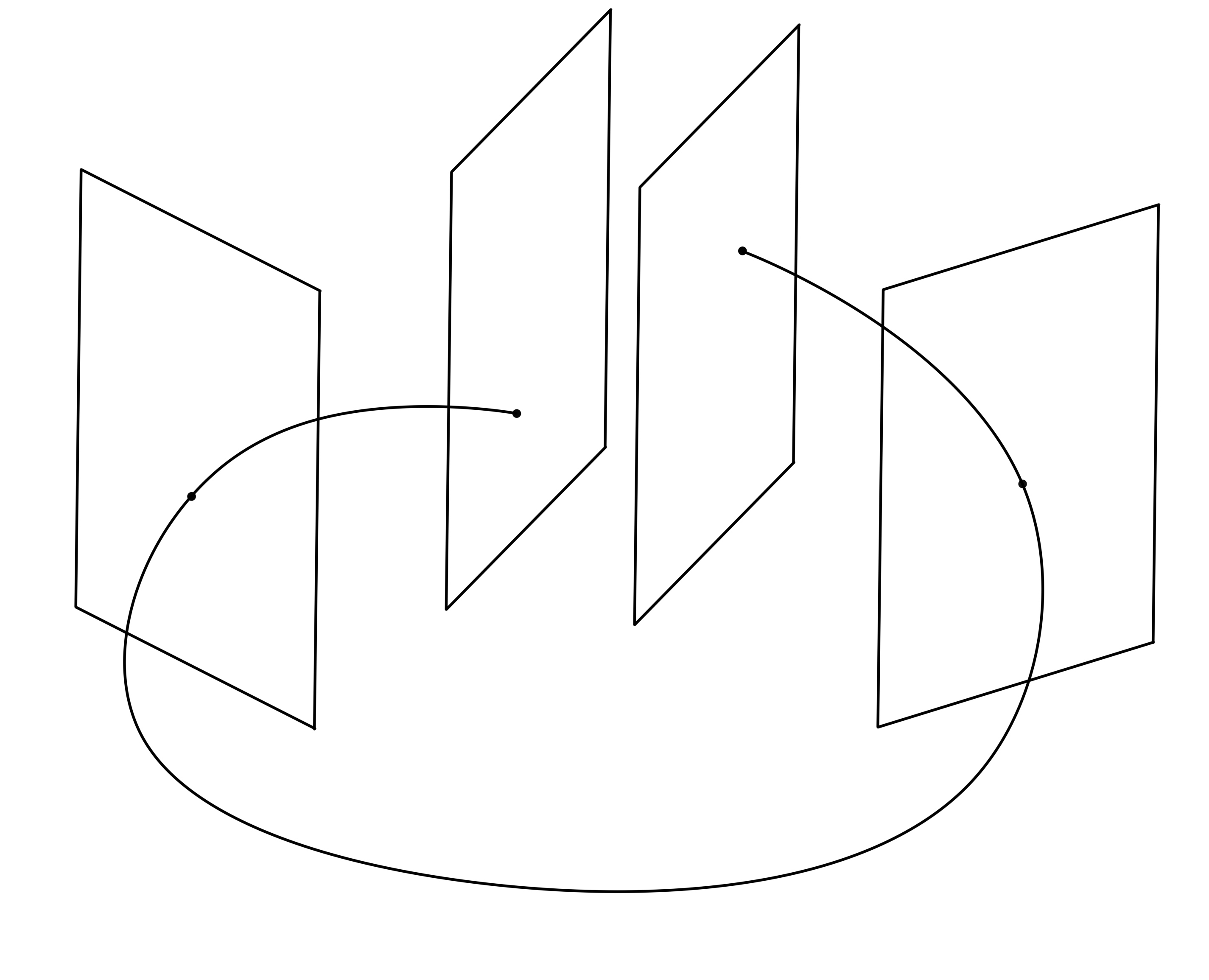}};
\node at (4.5,7.3) {$\phi(N)$};
\node at (5.2,7.3) {$\cong$};
\node at (5.7,7.33) {$N$};
\node at (5,5.25) {$x$};
\node at (3.7,3.7) {$\phi(x)$};
\end{tikzpicture}
\caption{Suspended diffeomorphism}
\label{sus}
\end{figure}
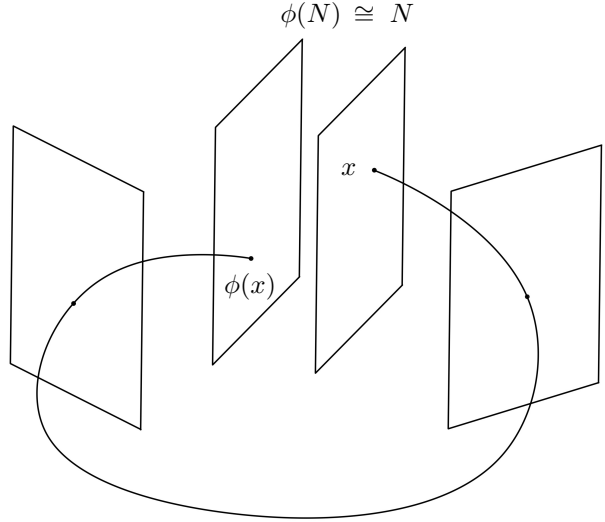
Indeed, let $\tilde N$ be the manifold defined as $ \tilde N:= N \times [0,1] / \sim $ where we identify $(x,0)$ with $(\phi(x),1)$. Consider the horizontal vector field $\partial_\theta$ on $N\times [0,1]$, where $\theta\in [0,1]$. This vector field immediately pushes down to another field on $\tilde N$ that we call $X$. Observe that $N$ is a cross section of $X$ and its (time-one) return map is conjugate to $\phi$. Arguing as before, we show that $X$ is geodesible and can be extended to a steady Euler flow on $\mathbb S^n$, $n\in \{3\dim\, N+5,3\dim\, N+6\}$, thus showing that $\phi$ is Euler-embeddable.

\subsection{Turing completeness}\label{SS5.2}

Let us first recall that a Turing machine is a $5$-tuple $(Q,q_0,F, \Sigma, \delta)$ where
\begin{itemize}
\item $Q$ is a finite non-empty set, the set of ``states''.
\item $q_0 \in Q$ is the initial state.
\item $F \in Q$ is the halting state.
\item $\Sigma$ is the alphabet, a finite set of cardinality at least two.
\item $\delta:(Q \backslash F) \times \Sigma \rightarrow Q \times \Sigma  \times \{L,N,R\}$ is a partial function called a transition function. We denote by $L$ the left shift, $R$ is the right shift and $N$ represents a ``no shift''.
\end{itemize}

Following Tao~\cite{T1}, we consider a Turing machine with a single tape that is infinite in both directions and a single halting state, with the machine shifting the tape rather than a tape head; in particular we do not need to isolate a blank symbol character in the alphabet (anyway, all the results here apply to other variants of a Turing machine). We denote by $q$ the current state, and $t=(t_n)_{n\in \mathbb{Z}}$ the current tape. For a given Turing machine $(Q,q_0,F,\Sigma,\delta)$ and an input tape $s=(s_n)_{n\in \mathbb{Z}}\in \Sigma^{\mathbb{Z}}$ the machine runs applying the following algorithm:
\begin{enumerate}
\item Set the current state $q$ as the initial state and the current tape $t$ as the input tape.
\item If the current state is $F$ then halt the algorithm and return $t$ as output. Otherwise compute $\delta(q,t_0)=(q',t_0',\varepsilon)$, with $\varepsilon \in \{L,R,N\}$.
\item Replace $q$ with $q'$ and $t_0$ with $t_0'$.
\item Replace $t$ by the $\varepsilon$-shifted tape, then return to step $(2)$.
\end{enumerate}
For any input the machine will halt at some point and return an output or run indefinitely. The Turing completeness of a dynamical system can be understood in terms of the concept of a \emph{universal Turing machine}, which is a machine that can simulate all Turing machines.

In \cite{CMPP2} we proved the following theorem:
\begin{theorem}\label{thm:main}
There exists an Eulerisable flow $X$ in $\mathbb S^3$ that is Turing complete in the following sense. For any integer $k\geq 0$, given a Turing machine $T$, an input tape $t$, and a finite string $(t_{-k}^*,...,t_k^*)$ of symbols of the alphabet, there exist an explicitly constructible point $p\in \mathbb S^3$ and an open set $U\subset \mathbb S^3$ such that the orbit of $X$ through $p$ intersects $U$ if and only if $T$ halts with an output tape whose positions $-k,...,k$ correspond to the symbols $t_{-k}^*,...,t_k^*$. The metric $g$ that makes $X$ a stationary solution of the Euler equations can be assumed to be the round metric in the complement of an embedded solid torus.
\end{theorem}

Now we prove Theorem~\ref{T.main2}, i.e., that there exists a steady Euler flow on $\mathbb S^{17}$ that is Turing complete using the former construction (Theorem \ref{main1} and its proof). Even if in Theorem~\ref{T.main2} the sphere is higher dimensional, the method of proof allows us to promote arbitrary Turing complete constructions to higher dimensional Eulerizable vector fields. The initial data in our proof can be seen as \emph{black boxes} that can be assembled to yield a whole family of ad-hoc brand-new constructions. A key feature of this theorem is that it is based on an $h$-principle which is algorithmic. In particular, to prove Theorem~\ref{T.main2} we take as initial datum an orientation-preserving diffeomorphism $\phi$ of the torus $\mathbb T^4$ that encodes a universal Turing machine in the following sense~\cite[Proposition 1.10]{T1}:

\begin{prop}\label{Prop:Turing}
There exists an explicitly constructible diffeomorphism $\phi:\mathbb T^4 \rightarrow \mathbb T^4$ such that for any Turing machine $(Q,q_0,F,\Sigma, \delta)$ there is an explicitly constructible open set $U_{t_{-n},...,t_n}\subset \mathbb T^4$ attached to each finite string $t_{-n},...,t_{n} \in \Sigma$, and an explicitly constructible point $y_s\in \mathbb T^4$ attached to each $s\in \Sigma^\mathbb{Z}$ such that the Turing machine with input tape $s$ halts with output $t_{-n},...,t_n$ in positions $-n,...,n$, respectively, if and only if the orbit $y_s,\phi(y_s),\phi^2(y_s),...$ enters $U_{t_{-n},...,t_{n}}$.
\end{prop}

Let us now prove Theorem~\ref{T.main2} using that the diffeomorphism $\phi:\mathbb T^4\to \mathbb T^4$ that encodes a universal Turing machine constructed in Proposition~\ref{Prop:Turing} can be Euler-embedded in $\mathbb S^{17}$, c.f. Corollary~\ref{Cor:difeo}. More precisely, let $N$ be the $5$-dimensional manifold defined as $N:= \mathbb T^4 \times [0,1] / \sim $, where we identify $(x,0)$ with $(\phi(x),1)$, and $X$ the vector field on $N$ obtained by pushing down the horizontal vector field $\partial_\theta$ on $\mathbb T^4\times [0,1]$ ($\theta\in [0,1]$). { This vector field is geodesible, so by Theorem \ref{main2} there exists a Reeb embedding $e:(N,X)\to \mathbb (S^{17}, \xi_{std})$, which is a stationary solution to the Euler equations $u$ on $\mathbb S^{17}$ of Beltrami type satisfying $u|_{e(N)}=X$}. The { embedding} $e$ is explicitly constructible because the applied $h$-principle is algorithmic; also the vector field $u$ is constructible, because it is the Reeb field of a defining contact form $\alpha$ of $\xi_{std}$, which is also algorithmic (see the proof of Theorem~\ref{main2}). { The fact that these embeddings are algorithmically constructible is a key feature to deduce the existence of a Turing complete Beltrami field. Indeed, the capacity of simulating a Turing machine requires that the initial point associated with an input of the machine is constructible (i.e., computable by a Turing machine). }

In view of the previous discussion, let us take a point $\tilde y_s\in \mathbb S^{17}$ as the image of the point $y_s\times \{0\}\in N$ under the embedding $e$, and a neighborhood $\tilde U_{t_{-n},...,t_n}\subset \mathbb S^{17}$ as a neighborhood in $\mathbb S^{17}$ of the image of the open set $U_{t_{-n},...,t_n} \times \{ 0 \}\subset N$ under the embedding $e$. Then, Theorem~\ref{T.main2} can be restated in a more precise way as follows:

\begin{theorem*}[Beltrami fields are Turing complete]
There exists a Beltrami field $u$ on $\mathbb S^{17}$ for some Riemannian metric $g$ such that for any Turing machine $(Q,q_0,F,\Sigma, \delta)$ there is an explicitly constructible open set $U_{t_{-n},...,t_n}\subset \mathbb T^4$ attached to each finite string $t_{-n},...,t_{n} \in \Sigma$ and an explicitly constructible point $y_s\in \mathbb T^4$ attached to each $s\in \Sigma^\mathbb{Z}$ such that the Turing machine with input tape $s$ halts with output $t_{-n},...,t_n$ in positions $-n,...,n$ respectively if and only if the trajectory of $u$ with initial datum $\tilde y_s$ enters $\tilde U_{t_{-n},...,t_n}$.
\end{theorem*}
\begin{Remark}
The metric $g$ in this theorem is the canonical metric of $\mathbb S^{17}$ in the complement of a neighborhood of $e(N)$. In fact, we observe that the dimension of the target sphere may be reduced by $2$ in all the applications of this section by considering a sphere with an overtwisted contact structure. In this case we would obtain a Turing complete Euler flow in $\mathbb S^{15}$, but we cannot longer guarantee that the metric $g$ is the canonical one in the complement of a neighborhood of $e(N)$. This can be easily fixed if one starts with the standard tight contact structure on the sphere and make it overtwisted on a small ball $U\subset\mathbb S^{15}$ using the relative $h$-principle for overtwisted contact structures \cite{BEM}, so that the contact form is the standard one in $\mathbb S^{15}\setminus U$. Then, taking an embedding $e:N\to \mathbb S^{15}$ with $e(N)\subset U$, Theorem~\ref{main2} can be applied with the overtwisted target manifold $M\equiv U$, and therefore the resulting Reeb embedding $\widetilde e$ isotopic to $e$ can be taken such that $\widetilde e(N)\subset U$ (although it is no longer $C^0$-close to $e$). The metric $g$ can then be chosen to be the canonical one outside $U$ because the contact form coincides with the standard one in $\mathbb S^{15}\setminus U$ by construction.
\end{Remark}

\subsection{The existence of a universal solution in $\mathbb R^{n}$}

Using the ideas developed in this article, we can show that there exists an Eulerisable flow in $\mathbb R^n$ which, in some sense, exhibits all possible lower-dimensional dynamics. To be more precise, let us introduce the following definitions:
\begin{defi}
Given two vector fields $X_1$ and $X_2$ in $N \times \mathbb S^1$, where $N$ is a compact manifold, we say that $X_1$ is \emph{$(\varepsilon,k)$-conjugate to $X_2$} if there is a diffeomorphism $\varphi:N\times \mathbb S^1 \rightarrow N \times \mathbb S^1$ such that
$$ ||\varphi_*X_1 - X_2 ||_{C^k(N\times \mathbb S^1)} < \varepsilon. $$
\end{defi}

\begin{defi}
Fix a positive integer $k$. A vector field $u$ in $\mathbb R^n$ is \emph{$N$-universal} if for any $\varepsilon$ and any vector field $X$ on $N$ there is an \emph{invariant} submanifold $\tilde N$ of $u$ diffeomorphic to $N\times \mathbb S^1$ such that $u|_{\tilde N}$ is $(\varepsilon,k)$-conjugate to $X+\partial_\theta$ with $\theta\in \mathbb S^1$.
\end{defi}

\begin{theorem}\label{Teo_univ}
Let $N$ be a compact manifold. There exists an $N$-universal Eulerisable flow of Beltrami type in $\R^n$, where the dimension is the smallest odd integer $n\in\{3\dim\, N+5,3\dim\, N+6\}$.
\end{theorem}
\begin{proof}
We first recall that the space $\mathfrak{X}(N)$ of smooth vector fields on $N$ is second countable with the Whitney topology~\cite[Chapter 2.1]{Hi}. In particular, it is separable and hence there is a countable set of vector fields $\{ X_j \}_{j\in \mathbb{Z}}$ which is dense in $\mathfrak{X}(N)$. For every pair $(N,X_j)$, we can take the suspension of the vector field $X_j$ as in Subsection~\ref{SS5} to obtain a countable set of pairs $(N_j, Y_j)$ where $N_j$ is diffeomorphic to $N \times \mathbb S^1$ and $Y_j:= X_j+\partial_\theta$ is a geodesible flow. Now take a countable collection of contact balls $(U_j,\xi_{std})\subset \mathbb R^n$ with pairwise disjoint closures, where $n$ is the smallest odd integer $n\in\{3\dim\, N+5,3\dim\, N+6\}$. By Theorem~\ref{main2} there exists an embedding $e_j$ of $(N_j, Y_j)$ for each $j\in\mathbb Z$ into $(U_j,\xi_{std})$ such that there is a defining contact form $\alpha_j$ whose Reeb field $R_j$ on $e_j(N_j)$ restricts to $Y_j$. Observe that we can take $\alpha_j=\alpha_{std}$, the standard contact form, in a neighborhood of the boundary $\partial U_j$. This allows us to define a smooth global contact form $\alpha$ on $\mathbb R^n$ by setting $\alpha:=\alpha_j$ on each $U_j$ and $\alpha:=\alpha_{std}$ on $\mathbb R^n\backslash \bigcup_{j\in\mathbb Z} U_j$; it is obvious that the Reeb field $R$ associated to $\alpha$ satisfies $R|_{e_j(N_j)}=Y_j$ for all $j$.

Fixing an integer $k$, it follows from the previous construction that for any vector field $X\in \mathfrak{X}(N)$ and any $\varepsilon>0$, there exists $j_0\in\mathbb Z$ { so that $X_{j_0}$ is $(\varepsilon,k)$-conjugate to $X$.} Then $e_{j_0}(N_{j_0})$ is an invariant submanifold of $R$, and $R|_{e_{j_0}(N_{j_0})}=Y_{j_0}$. Accordingly, $Y_{j_0}$ is $(\varepsilon,k)$-conjugate to $X+\partial_{\theta}$. Since any Reeb field is an Eulerisable flow of Beltrami type (c.f. Section~\ref{SS.Belt}), the theorem follows.	
\end{proof}

The method of proof of Theorem~\ref{Teo_univ} allows us to provide a different proof of a theorem of Etnyre and Ghrist in~\cite{EG2}. Specifically, we can show that there exists an Eulerisable flow in $\R^3$ exhibiting periodic integral curves of all possible knot and link types; when the Riemannian metric of $\mathbb R^3$ is fixed and analytic, this result was proved in~\cite{EP1}.
\begin{Corollary}
There exists an Eulerisable flow of Beltrami type in $\R^3$ exhibiting stream lines of all possible knot and link types.
\end{Corollary}
\begin{proof}
The set of all knot and link types of smoothly embedded circles in $\mathbb R^3$ is known to be countable. Let us now embed a representative $L_j$ of each knot and link type in pairwise disjoint Darboux balls $(U_j,\xi_{std})\subset\mathbb R^3$ as in the proof of Theorem~\ref{Teo_univ}. Then, for all $j\in\mathbb Z$, there is an isotopy of the link $L_j$, $C^0$-close to the identity, which makes it positively transverse to $\xi_{std}$, see e.g.~\cite{Ge}. For the ease of notation, we still denote the deformed link by $L_j$. Applying Lemma~\ref{L.Inaba} to each $L_j$ endowed with the vector field $X_j:=\partial_\theta$, where $\theta\in\mathbb S^1$ parametrizes $L_j$, we conclude that there is a contact form $\alpha$ in $\mathbb R^3$ whose Reeb vector field contains periodic orbits of all possible knot and link types. (Note that the condition that $X_j$ preserves $TL_j\cap \xi_{std}$ is trivially satisfied in this case.) The statement then follows using the correspondence between Reeb flows and Beltrami fields.
\end{proof}

\subsection{Even dimensional Euler flows}
In all the constructions that we have used to prove Theorems~\ref{T.main1} and~\ref{T.main2}, the ambient manifold is odd dimensional because we exploit the connection between hydrodynamics and contact geometry. We finish this section with a result that allows us to establish the universality of the Euler flows also for even dimensional ambient manifolds. The main observation, which is the even dimensional analog of Theorem~\ref{main1}, is the following proposition:

\begin{prop}\label{prop:even}
Let $N$ be a compact manifold endowed with a geodesible flow $X$. Then there exists an embedding $e:(N,X)\rightarrow \mathbb S^{n}\times \mathbb S^1$ with $n$ the smallest odd integer $n\in\{3\dim\, N+2,3\dim\, N+3\}$, and an Eulerisable field $u$ on $\mathbb S^{n}\times \mathbb S^1$ such that $u|_{e(N)}=X$.
\end{prop}

\begin{proof}
Applying Theorem~\ref{main2} we obtain an embedding $\tilde e:(N,X)\rightarrow \mathbb (S^n,\xi_{std})$ and a defining contact form $\tilde \alpha$ whose Reeb vector field $R$ restricts to $X$ on $\tilde e(N)$. By the correspondence Theorem~\ref{correspondence}, the field $R$ is a Beltrami field with constant proportionality factor for some Riemannian metric $\tilde g$ on $\mathbb S^n$. Consider now the $(n+1)$-manifold $S^{n}\times S^1$ endowed with the Riemannian metric $g:= \tilde g + d\theta^2$, $\theta\in\mathbb S^1$, and define the trivial extension of the Reeb flow $R$ as the vector field $u:=(R,0)$ on $\mathbb S^n\times \mathbb S^1$. The dual $1$-form of $u$ using the metric $g$ is
$$\alpha= \iota_u g= \pi^*\tilde\alpha\,, $$
where $\pi$ is the canonical projection $\pi:\mathbb S^{n}\times \mathbb S^1 \rightarrow \mathbb S^{n}$. Accordingly, $\iota_ud \alpha=\pi^*(i_Rd\tilde\alpha)=0$ and $u$ preserves the (Riemannian) volume form $\mu=\mu_{\tilde g}\wedge d\theta$. Defining the embedding $e:N\rightarrow \mathbb S^n\times \mathbb S^1$ of $N$ as $e:=i\circ \tilde e$, where $i$ is the natural inclusion of $\mathbb S^n$ into $\mathbb S^n\times \mathbb S^1$, we conclude that $u$ is a steady Euler flow on $\mathbb S^n\times \mathbb S^1$ such that $u|_{e(N)}=X$.
\end{proof}

The proof of Theorems~\ref{T.main1} and~\ref{T.main2} for even dimensional ambient manifolds is then the same, mutatis mutandis, as in Subsections~\ref{SS5} and~\ref{SS5.2}, but invoking Proposition~\ref{prop:even} instead of Theorem~\ref{main2}.

\section{Flexibility of Reeb embeddings}\label{S4}

The goal of this section is to prove the Reeb embedding Theorem~\ref{main2} and provide some generalizations that can be useful for further applications in Contact Geometry. The proof of this result follows the usual pattern in the $h$-principle theory:
\begin{enumerate}
\item We first define a purely topological condition that an embedding needs to satisfy and introduce the concept of \emph{formal iso-Reeb embedding} (Definition \ref{fR}).
\item As it is customary in the $h$-principle theory (see e.g.~\cite{Mur, BEM, LEng}), we restrict ourselves to a particular subclass of formal iso-Reeb embeddings called \emph{small formal iso-Reeb embeddings} (Definition \ref{sfR}), and prove that any small formal iso-Reeb embedding can be deformed into a genuine (small) iso-Reeb embedding (Theorem~\ref{hp2}).
\item Finally, we check under which conditions a given embedding can be equipped with a small formal iso-Reeb embedding structure and show that for embeddings of high enough codimension we can always find such a formal structure, see Lemma~\ref{lemmaHo}. These dimensional restrictions account for the bounds in Theorem~\ref{main2}.
\end{enumerate}

These results are presented as follows. In Subsection~\ref{SS.basics} we introduce some basic notions of the $h$-principle that are used along this section. A technical stability lemma for vector bundles which is instrumental for the next subsections is presented in Subsection~\ref{SS.classic}. In Subsection~\ref{SS.Reebh} we introduce the definitions of formal iso-Reeb embedding and small iso-Reeb embedding, and prove a full $h$-principle in this context (Theorems~\ref{hp1} and~\ref{hp2}). The key lemma to establish the existence of formal iso-Reeb embeddings of high codimension is presented in Subsection~\ref{SS.key}. Finally, using this machinery we prove Theorem~\ref{main2} in Subsection~\ref{SS.fin}.

\subsection{Basic notions of the $h$-principle} \label{SS.basics}
Let us introduce some basic notions in the $h$-principle theory which are key to provide precise statements.

Fix a smooth fibration $\pi: X\rightarrow V$. Denote by $\pi^r:J^r(X)\rightarrow V$ the associated $r$-jet fibration. There is also a natural projection $p_r : J^r(X) \rightarrow X$. Given a section $\sigma:V\rightarrow X$, denote by $j^r(\sigma): V \rightarrow J^r(X)$ the canonical $r$-jet extension. Thus, we have a natural inclusion $\operatorname{Sec}(V,X)\rightarrow \operatorname{Sec}(V,J^r(X))$ where $\operatorname{Sec}(V,X)$ and $\operatorname{Sec}(V,J^r(X))$ are the spaces of sections from $V$ to $X$ and $J^r(X)$ respectively.

A subset $\mathcal{R} \subset J^r(X)$ is called a partial differential relation of order $r$. Define $\operatorname{Sec}_\mathcal{R}(V,J^r(X)) \subset \operatorname{Sec}(V,J^r(X))$ as the space of formal solutions. It is defined as the space of sections satisfying that the image of the section lies in $\mathcal{R}$. Moreover, define the space of solutions, and denote it by $\operatorname{Sec}_\mathcal{R}(V,X) \subset \operatorname{Sec}(V,X)$, to be the space of sections whose $r$-jet extension is a formal solution. A solution in $\operatorname{Sec}_\mathcal{R}(V,X)$ is called a holonomic solution.

\begin{defi}
We say that a partial differential relation $\mathcal{R}$ obeys the rank $k$ $h$-principle if the inclusion $e:\operatorname{Sec}_\mathcal{R}(V,X) \rightarrow \operatorname{Sec}_\mathcal{R}(V,J^r(X))$ of the space of solutions into the space of formal solutions, which induces morphisms $\pi_j(e):\pi_j(\operatorname{Sec}_\mathcal{R}(V,X)) \rightarrow \pi_j(\operatorname{Sec}_\mathcal{R}(V,J^r(X)))$, satisfies that $\pi_j(e)$ is an isomorphism for $j\leq k$. If $k=\infty$ we say that $\mathcal{R}$ satisfies \emph{the full $h$-principle}.
\end{defi}

The following terminology is standard. We say that $\mathcal{R}$ satisfies a:
\begin{itemize}
\item \emph{parametric $h$-principle} if we can deform formal solutions by holonomic solutions parametrically.
\item  \emph{relative parametric $h$-principle} if the following holds: Fix a closed subset $C \subset K$, where $K$ is any compact parameter space. Assume we have a family of formal solutions $\sigma_k, k\in K$ such that $\sigma_k$ with $k\in C$ is a holonomic solution. Then there exists a parametric family of formal solutions $\tilde \sigma_{k,t}, t\in [0,1]$ such that $\tilde \sigma_{k,0}=\sigma_k$, $\tilde \sigma_{k,1}$ are holonomic solutions and moreover $\tilde \sigma_{k,t} = \sigma_k$ for $k\in C$ and all $t$.

\item  \emph{relative to the domain $h$-principle} if the following is satisfied: For any closed subset $D\subset V$, assume we have a formal solution $\sigma$ that is holonomic in an open neighborhood $U$ of $D$. Then there exists a family of formal solutions $\sigma_t, t\in I$ such that $\sigma_0=\sigma$, $\sigma_1$ is holonomic and $\sigma_t|_U=\sigma|_U$ for all $t$.

\item  \emph{$C^0$-dense $h$-principle} if any formal solution $s: V \rightarrow J^r(X)$ can be approximated by a holonomic solution $j^r(\tilde \sigma)$ such that $p_r(s)$ is $C^0$-close to $\tilde \sigma$.

\end{itemize}

It is known (see for instance~\cite[Chapter 6]{hprinc}) that any partial differential relation that satisfies an $h$-principle: parametric, relative to the parameter, relative to the domain, actually satisfies a full $h$-principle.

\subsection{Classical stability lemmas for vector bundles}\label{SS.classic}
The following technical results will be used in the next subsections. Proofs are provided for the sake of completeness though they are well known to experts (see e.g.~\cite[Corollary 4.6]{Hus}).

\begin{lemma}\label{L:Vkt}
Let $V_{k,t}$ be a parametric family of complex bundles over a fixed smooth manifold $M$ with parameters given by $(k,t) \in K \times [0,1]$. Then, there exists a family of complex isomorphisms $\phi_{k,t}: V_{k,0} \to V_{k,t}$.
\end{lemma}
\begin{proof}
Take a finite number of sections $\sigma_{r}^{k,t}: M \to V_{k,t}$, $r\in \{1, \ldots, n\}$ varying continuously with the parameters such that for any point $p\in M$ and any parameter value $(k,t)$, there are $l:= \operatorname{rank} V_{k,t}$ sections $\sigma_{r}^{k,t}, 1\leq r\leq l$ (relabeling the index $r$ if necessary) defining a framing of the fiber over $p$. Then the bundle map:
\begin{eqnarray*}
p_{k,t}: \C^n \times M &\to & V_{k,t} \\
(\lambda_1, \ldots, \lambda_n, p) &\to &(\Sigma_{r=0}^n \lambda_r \sigma_{r}^{k,t}(p),p)
\end{eqnarray*}
is an epimorphism of vector bundles. By choosing a metric on each bundle, we find the adjoint map $p_{k,t}^*: V_{k,t} \to \C^n \times M$ that is a monomorphism of vector bundles.

So we may assume that $V_{k,t} \subset \C^n$. Now, fix an hermitian metric on $\C^n$. Denote by $H_{k,t}$ the orthogonal
to $V_{k,t}$ with respect to the fixed metric. Define a map
$$ f_{k,t}^{\varepsilon}: V_{k,t} \to V_{k,t+\varepsilon} $$
in the following way. Choose for each $p\in M$ and $v\in V_{k,t}$ the unique intersection point in $\C^n$ of the affine subspaces $v+H_{k,t}$ and $V_{k,t+\varepsilon}$ and denote it by $v_{\varepsilon}$. We define then the map as $f_{k,t}^\varepsilon(v)=v_{\varepsilon}$.
Finally, for each $p\in M$ define $X_{k,t}:= \lim_{\varepsilon \to 0} \frac{f_{k,t}^{\varepsilon}(v)-v}{\varepsilon}$. This defines a time dependent vector field over each fiber $\{ p \} \times \C^n$. Clearly, its associated flow $\phi_{k,s}$ satisfies that $\phi_{k,s}(V_{k,0}) = V_{k,s}$ and it is an isomorphism of complex bundles, { by construction of $X_{k,t}$ using $f_{k,t}^\varepsilon$}.
\end{proof}

\begin{Corollary} \label{coro:formalDarboux}
Let $(V_{k,t}, [\omega_{k,t}])$ be a parametric family of conformal symplectic bundles over a fixed smooth manifold $M$ with parameter given by $(k,t) \in K \times [0,1]$. Then, there exists a family of isomorphisms $\phi_{k,t}: V_{k,0} \to V_{k,t}$ which furthermore are conformal symplectomorphisms.
\end{Corollary}
\begin{proof}
Since in this paper we only consider conformal symplectic structures induced on contact distributions that are cooriented, we restrict to this case (the general case can be easily reduced to this one by a finite covering argument). In particular, we may assume that the conformal symplectic structure is induced by a symplectic structure $\omega_{k,t}$. Then, fix compatible complex structures $J_{k,t}$. This can be done continuously in families since the space of complex structures which are compatible with a fixed symplectic structure is contractible and thus, we can  always find  global sections: i.e. almost complex structures in the bundle, also in parametric families. This produces an hermitian metric $h_{k,t}$ on $V_{k,t}$.

Extend $h_{k,t}$ to a global hermitian structure $\tilde{h}_{k,t}$ in $\mathbb C^n$. We can then mimic the proof of Lemma~\ref{L:Vkt} to obtain a family of hermitian preserving isomorphisms, which are in addition conformal symplectomorphisms (and in fact symplectomorphisms for the chosen $\omega_{k,t}$).
\end{proof}

Adapting the proof for the real case, we obtain:
\begin{lemma} \label{lem:real}
Let $V_{k,t}$ be a parametric family of real bundles over a fixed smooth manifold $M$ with parameters given by $(k,t) \in K \times [0,1]$. Then, there exists a family of real isomorphisms $\phi_{k,t}: V_{k,0} \to V_{k,t}$.
\end{lemma}

\subsection{An $h$-principle for iso-Reeb embeddings}\label{SS.Reebh}
Following previous notation, let $X$ be a geodesible vector field on $N$, and denote by $\beta$ the $1$-form such that $\eta=\ker \beta$ and $\beta(X)=1$. { That is, $\eta$ is the transverse hyperplane field preserved by $X$.} Let $(M,\xi)$ be a contact manifold with defining contact form $\alpha$, i.e. $\ker \alpha=\xi$.

\begin{Remark}
As in previous sections we either assume that $N$ is compact or $N$ is properly embedded into $M$.
\end{Remark}

With a slight abuse of notation, we will denote $\alpha \circ F_1$ for $\alpha(F_1(\cdot))$ and $d\alpha \circ F_1$ for $d\alpha(F_1(\cdot),F_1(\cdot))$. This is also denoted by ${F_1}^*\alpha$ and ${F_1}^*d\alpha$ in similar discussions in~\cite{hprinc}.

\begin{defi}\label{fR} An embedding $f:(N,X,\eta)\rightarrow (M,\xi)$ is a \emph{formal iso-Reeb} embedding if there exists a homotopy of monomorphisms
$$ F_t:TN \longrightarrow TM,  $$
such that $F_t$ covers $f$, $F_0=df$, $h_1 \alpha \circ F_1= \beta$ and $d\beta|_{\eta}= h_2 d\alpha \circ F_1|_{\eta} $ for some strictly positive functions $h_1$ and $h_2$ on $N$.

\end{defi}
It is clear that any genuine iso-Reeb embedding is a formal iso-Reeb embedding. Indeed, take an iso-Reeb embedding $e:(N,X,\eta)\rightarrow (M,\xi)$, so by hypothesis we have $e^*\alpha=\beta$, which reads as $\alpha \circ  de =\beta$. Thus, we also obtain $e^*d \alpha=d \beta$ that restricted to $\eta$ can be written as  $d\beta|_{\eta}= d\alpha \circ F_1|_{\eta}$, and it is clear that $(e,F_t=de)$ is a formal iso-Reeb embedding.

Both conditions $h_1 \alpha \circ F_1= \beta$ and $d\beta|_{\eta}= h_2 d\alpha \circ F_1|_{\eta} $ are required to fix the definition of formal iso-Reeb embedding. One may be tempted to say that the first condition naturally implies the second one, but this is tantamount to saying that $F_1$ commutes with the exterior differential. This only holds when $F_1$ is holonomic, i.e. the pull-back (through the differential of a morphism) commutes with the exterior differential.

The first main result of this subsection is a full $h$-principle for iso-Reeb embeddings into overtwisted contact manifolds. The general case is more elaborated because it involves the introduction of a particularly appropriate subclass of iso-Reeb embeddings, and will be discussed later.

\begin{theorem}[$h$-principle for iso-Reeb embeddings into overtwisted manifolds]\label{hp1}
Let $f:(N,X,\eta)\rightarrow (M,\xi)$ be a formal iso-Reeb embedding with formal differential $F_t$ such that $\dim N<\dim M$. Furthermore, assume that $\xi$ is an overtwisted contact structure. Then, there exists a homotopy $(f^s, F_t^s)$ of formal iso-Reeb embeddings such that $(f^0, F_t^0)= (f, F_t)$ and such that $(f^1, F_t^1)=(f^1, df^1)$ is a genuine iso-Reeb embedding. Moreover, the natural inclusion of the space of iso-Reeb embeddings whose image does not intersect a fixed overtwisted disk $\Delta$ into the space of formal iso-Reeb embeddings whose image does not intersect $\Delta$ on a fixed overtwisted contact manifold is a homotopy equivalence.
\end{theorem}

\begin{proof}
All the bundles in the next paragraph are bundles over $N$, i.e. $TM$, $TN$, $\xi$, etc. are to be understood as the restriction over $N$ of these bundles, but we shall omit notations like $TM|_N$ for the sake of simplicity.
\\

{\noindent \bf {Step 1: Deform $\xi$ to a homotopic formal contact structure $\bar \xi_1$ on $N$ for which $F_0(\eta) \subset \bar \xi_1$.}} It is standard to find a family of isomorphisms $G_t: TM \to TM$ such that $G_0=id$ and $G_t \circ F_0= F_t$. Denote $\bar \xi_t:={G_t}^{-1}(\xi)$, so we have $\bar \xi_0= \xi$. Define $\omega_t:= d\alpha \circ G_t$ that equips $\bar \xi_t$ with a symplectic vector bundle structure $(\bar \xi_t, \omega_t)$ such that, for $t=1$ we obtain $F_0(\eta) \subset \bar \xi_1$. Denote by $\beta$ a defining $1$-form for $\eta$. Then
$$ (\omega_1)|_{\eta} = d\alpha \circ G_1|_{\eta}= d\alpha \circ F_1 =  (h_2)^{-1} d\beta, $$
where the last equality comes from the definition of formal iso-Reeb embedding. Up to conformal transformation, we can assume that $(\omega_1)|_\eta= h_2(d\alpha \circ F_1)=d\beta$. Therefore, $\bar \xi_1$ is a formal contact structure, homotopic to $\xi$, such that $F_0(\eta)\subset \bar \xi_1$.
\\

{\noindent \bf Step 2: Extend $\bar \xi_1$ to a contact structure on a neighborhood of $N$ and make the inclusion an iso-Reeb embedding.}
Extend the family of distributions $\bar \xi_t$ that are defined over $N$ to a family of distributions $\tilde \xi_t$ defined over a neighborhood $\mathcal O p(N)$. A possible way to do this is just to extend the isomorphisms $G_t: TM \to TM$ over $N$ to a new family $\tilde G_t: TM \to TM$ over $\mathcal{O}p(N)$ that can be used to define $\tilde \xi_t:=\tilde G_t (\tilde \xi_0).$ Then, using Lemma~\ref{tech} where $(M,\xi)$ is the neighborhood $\mathcal{O}p(N)$ and $\tilde \xi_1$, we obtain a contact structure $\hat \xi_1$ that is defined on $\mathcal{O}p(N) \supset N$ inside $M$ and is homotopic to $\tilde \xi_1$. Also, we obtain an iso-Reeb embedding of $(N, X, \eta)$  with respect to a contact form $\tilde \alpha_1$ defining the contact structure $\hat \xi_1$.
\\

{\noindent \bf Step 3: Reduce to formal isocontact embeddings.}
Summarizing, we have obtained that $\tilde \xi_0 =\xi$ and $\hat \xi_1$ are homotopic as formal contact structures in the neighborhood of $N$. By Corollary~\ref{coro:formalDarboux}, we can find a family of bundle isomorphisms $\phi_t: \tilde \xi_1 \to \tilde \xi_t$ that preserves the conformal symplectic structures on a small neighborhood of $N$. Extend $\phi_t$ to $TM|_{\Op (N)}$ and define the family $(e=id, H_t= \phi_t)$ with $e=id: \Op (N) \to \Op (N)$. It is a codimension $0$ formal isocontact embedding.
\\

{\noindent \bf Step 4: Conclusion.}
Applying the $h$-principle for isocontact embeddings in codimension $0$ with overtwisted target, c.f.~Theorem \ref{CorBEM}, we obtain the first part of Theorem \ref{hp1}.

Now observe that the previous arguments work parametrically. Also, it is simple to check that the proof is relative to any closed subdomain of the domain $N$. It is left to check that it works relative to the parameter, however this is not true in general. It is simple to realize that a sufficient condition to reproduce the proof making it relative to the parameter, see~\cite{BEM}, is restricting to the class of embeddings which do not intersect a fixed overtwisted disk. This is because in the previous construction we naturally obtain genuine iso-Reeb embeddings which do not intersect a fixed overtwisted disk. It is clear that for this subclass the previous three properties, parametric, relative to the domain and relative to the parameter, imply a full $h$-principle. The theorem then follows.
\end{proof}



Let us consider now a specific subclass of iso-Reeb embeddings, what we call \emph{small iso-Reeb embedding}. While it imposes an extra condition on the iso-Reeb embedding, the advantage is that it will allow us to prove a full $h$-principle.

\begin{defi} \label{sfR} Assume that there is a decomposition $(\xi|_N, d\alpha|_N) = (\xi' \oplus V, d\alpha|_{\xi'} + d\alpha|_{V})$ as orthogonal conformal symplectic subbundles { (i.e., a product of conformal symplectic bundles)}, and we further assume that $V$ is a proper subbundle of $\xi$.

An embedding $f:(N,X, \eta )\rightarrow (M,\xi=\ker \alpha)$ is a \emph{small formal iso-Reeb embedding} if there exists a homotopy of monomorphisms
$$ F_t:TN \longrightarrow TM\,,  $$
such that $F_t$ covers $f$, $F_0=df_0$,  $F_1(\eta)\subsetneq \xi'$ and $d\beta|_{\eta}= h_2 d\alpha \circ F_1|_{\eta} $, for some strictly positive function $h_2$ on $N$.

Likewise we say that $f:(N,X, \eta )\rightarrow (M,\xi=\ker \alpha)$ is a \emph{small iso-Reeb embedding} if $df(\eta)= TN \cap \xi$ and $df(\eta)\subsetneq \xi'$, where $\xi=\xi' \oplus V$ is an orthogonal conformal symplectic decomposition and $V$ is a non trivial subbundle.
\end{defi}

Clearly, any small iso-Reeb embedding is in particular an iso-Reeb embedding.  The embedding satisfies that $\xi \cap TN=\eta$, and hence by Lemma~\ref{L.Inaba} there is a contact form such that its Reeb vector field satisfies $R|_N=X$. If $X$ is negatively transverse to $\xi$, one can consider $-X$ instead. Otherwise, the contact form such that a negatively transverse $X$ is Reeb is a negative contact form.

\begin{theorem}[$h$-principle for small iso-Reeb embeddings]\label{hp2}
Let $f: (N,X,\eta) \rightarrow (M,\xi)$ be a small formal iso-Reeb embedding into a contact manifold with formal derivative $F_t$. Then there is a homotopy $(f^s, F_t^s)$ such that $(f^1, F^1_t=df^1)$ is a genuine (small) iso-Reeb embedding and one can take $f^s$ to be arbitrarily $C^0$-close to $f$.

Moreover the natural inclusion of the space of small iso-Reeb embeddings into the space of small formal iso-Reeb embeddings on a fixed contact manifold is a homotopy equivalence.
\end{theorem}

{ Notice that a parametric family of small formal iso-Reeb embeddings comes equipped with a parametric symplectic orthogonal decomposition of $\xi$ along the embeddings as in Definition \ref{sfR}. Hence the two subbundles of the decomposition have constant rank along the parametric family of embeddings.}

\begin{Remark}
Observe that an $h$-principle in general cannot be satisfied: if we take $(N,X)$ with $X$ a Reeb vector field and associated hyperplane distribution a contact structure $\xi'$, then an iso-Reeb embedding is equivalent to an isocontact embedding. It is well known that codimension-$2$ isocontact embeddings do not satisfy the $h$-principle. The inclusion of formal isocontact embeddings into genuine isocontact embeddings is not injective~\cite{CE}.
\end{Remark}
\begin{proof}[Proof of Theorem \ref{hp2}] \hfill \newline

{\bf \noindent Step 1: Deform the pair $\xi' \subset \xi$ to a new pair of formal contact structures $\bar \xi_1' \subset \bar \xi_1$ such that $F_0(\eta) \subset \bar \xi_1'$.}
We start by fixing the small formal iso-Reeb embedding $f$. Find $G_t: TM \to TM$ a family of isomorphisms such that $G_0=id$ and $G_t \circ F_0= F_t$. Denote $\bar \xi'_t:= {G_t}^{-1}(\xi')$ and $\bar \xi_t:={G_t}^{-1}(\xi)$, so we have $\bar \xi_0= \xi$. Define $\omega_t:= d\alpha \circ G_t$ that equips $\bar \xi_t$ with a conformal symplectic vector bundle structure $(\bar \xi_t, \omega_t)$ such that, for $t=1$ we obtain $F_0(\eta) \subset \bar \xi_1$. Likewise we obtain a conformal symplectic vector subbundle structure $\omega_t'= (d\alpha)|_{\xi'} \circ G_t$. Denote by $\beta$ the defining $1$-form for $\eta$, i.e. $\ker \beta= \eta$. We have
$$ (\omega_1)|_{\eta} = d\alpha|_{\xi'} \circ G_1|_{\eta}= d\alpha|_{\xi'} \circ F_1 =  (h_2)^{-1} d\beta\,, $$
where the last equality comes from the definition of formal small iso-Reeb embedding. Up to conformal transformation, we may assume that $(\omega_1)|_\eta= h_2(d\alpha \circ F_1)=d\beta$. We also obtain $ h_1(\alpha \circ F_1)=\beta$.
\\

{\bf \noindent Step 2: Find a positive codimension contact submanifold on a neighborhood of $N$ that contains it.}
Since, by hypothesis, there is a conformal symplectic orthogonal decomposition $(\bar \xi_1, \omega_1)= (\bar \xi_1', \omega_1') \oplus (\bar \xi_1')^{\perp \omega_1}$,  consider a  small neighborhood of the zero section of the bundle $\bar \xi'_t \to N$ (that exists because $F_0(\eta)$ is included but not equal to $\bar \xi'_1$), and denote it by $E_t$. Build an embedding of codimension (greater or equal than) $2$, $E_t \supset N$, by fixing a metric and applying the exponential map. Extending the exponential map to $(\bar \xi_1')^{\perp \omega_1}$, we obtain a local fibration of a neighborhood of $N$ as $\mathcal{O}p(N) \to E_t$, with linear conformal symplectic fiber given by $\bar \xi_1'$.
Thus, the conclusion is that the neighborhood $\mathcal{O}p(N)$ can be understood as a small tubular neighborhood of the formal contact submanifold $E_t$.
\\

{\bf \noindent Step 3: Mimic the proof of Theorem \ref{hp1}.}
We apply steps 2 and 3 as in  Theorem \ref{hp1} to obtain a contact structure $\tilde \xi'_1$ in $E_1$ and by~Lemma \ref{tech} an iso-Reeb embedding of $(N, X, \eta)$ into $(E_1, \tilde \xi')$. Using that $(E_1,  \tilde \xi'_1) \stackrel{id}{\to} (M, \xi)$ is a positive codimension formal isocontact embedding with open source manifold, we can apply the $h$-principle Theorem~\ref{isoc} to obtain an isocontact embedding, whose restriction to $N$ is $C^0$-close to the original embedding. To obtain the $C^0$-closeness we use the fact that we are just obtaining $C^0$-closeness on a positive codimension core, i.e. $N$, of the manifold $E_1$.  All the previous constructions can be done parametrically, relative to the parameter and relative to the domain. Accordingly, we obtain a full $C^0$-dense $h$-principle.
\end{proof}

Note that the data of a formal (small) iso-Reeb embedding include the choice of a distribution $\eta$ invariant under the flow of $X$. It is important to realize that this choice is not unique. In particular, the space of invariant distributions for a fixed geodesible vector field is a vector space, the transverse ones conforming a cone inside it. Moreover, Theorems~\ref{hp1} and~\ref{hp2} depend on the invariant distribution chosen, as the following result illustrates.

\begin{prop}
For an isocontact embedding $e:(N^{2n_0+1},\xi_N)\rightarrow (M,\xi_M)$ of codimension $2$ (which is clearly an iso-Reeb embedding for any Reeb vector field on $N$), there is a Reeb field $R$ and a distribution $\eta'$ invariant under $R$, which is  $C^0$-close to $\xi_N$, such that $(N,R,\eta')$ does not admit an iso-Reeb embedding into $M$, if $n_0\geq 2$.

\end{prop}

\begin{proof}
It is standard that one can take a Reeb field $R$ on $N$ with a standard neighborhood around a periodic orbit of type $\mathbb S^1\times D^{2n_0}$ endowed with a contact form $\alpha= d\theta + r^2 \alpha_{std}$, where $r$ is the radial coordinate on $D^{2n_0}$ and $\alpha_{std}$ is the standard contact form on $\mathbb S^{2n_0-1}$. In particular, the Reeb field has the form $\partial_\theta$ in this neighborhood. Now choose function $f:[0,1] \rightarrow \R^+$ satisfying the following conditions:
\begin{itemize}
\item $f(r)=0$ for $r \leq \frac{1}{2}$,
\item $f(r)$ is $r^2$ for $r \in [\frac{3}{4},1]$.
\end{itemize}
The form $\beta:=d\theta + f(r)\alpha_{std}$ extends to the whole manifold since it coincides with $\alpha$ on the boundary of the neighborhood. Moreover, it defines a transverse distribution $\eta':=\ker \beta$ that is invariant under the flow of $R$. Assume that the triple $(N,R,\eta')$ admits an iso-Reeb embedding $e'$ in $(M,\xi)$. Then the submanifold $\{0 \} \times D^{2n_0} \subset \mathbb S^1 \times D^{2n_0} \subset N \xrightarrow{e'} (M,\xi)$ is clearly a submanifold tangent to $\xi_M$, which leads to a contradiction.
\end{proof}

\subsection{A technical lemma: the $j$-connectedness of the space of isotropic subbundles inside a symplectic bundle}\label{SS.key}

The main result of this subsection is Lemma~\ref{lemmaHo} below. It is an instrumental lemma that will be our main tool to check that any smooth embedding with high enough codimension is a small formal iso-Reeb embedding. This is the most delicate point of the proof of Theorem~\ref{main2}. Throughout this subsection, the dimension of $N$ is denoted by $n$ and the dimension of $M$ is denoted by $2m+1$.

\begin{lemma}\label{lemmaHo}
Let $(N,X,\eta) \hookrightarrow (M,\xi)$ be an embedding such that $X \pitchfork \xi$, and $(\xi,\omega)$ is a symplectic hyperplane bundle of real rank $2m$. Denote by $\beta$ a defining $1$-form of $\eta$ in $N$.  If $2m \geq 3n -1$ then there exists a family $( \xi_t,  \omega_t)$ of symplectic distributions such that  $( \xi_1 ,  \omega_1)=(\xi,\omega)$ and $( \xi_0,  \omega_0)$ satisfies $\eta =  \xi_0 \cap TN$ and $\eta$ is an isotropic subspace of $\xi_0$. Furthermore $ ( \xi_t,  \omega_t)$ coincides with $(\xi, \omega)$ away from a neighborhood of $N$.
\end{lemma}

\begin{proof}
It is clear by assumption that $ TM|_N= \langle X \rangle \oplus \xi$. This implies that $\xi$ and $TN$ are transverse subspaces in $TM|_N$ and thus we can define a new bundle $\eta_1:= \xi \cap TN$. The linear interpolation between $\eta=\eta_0$ and $\eta_1$, which is well defined since $\eta_0$ and $\eta_1$ are contained in $TN$ and are transverse to $X$, provides a homotopy of subbundles between these two subbundles inside $TM|_N$. Denote this homotopy by $e_t: \eta_t \rightarrow TM$. Fix an auxiliary metric on $TM$ satisfying that  $\xi$ is orthogonal to $X$. Define $\xi_t= \eta_t \oplus TN^\perp$, so clearly $\xi_1=\xi$. We apply Lemma~\ref{lem:real} to obtain $G_t: \xi_1 \rightarrow \xi_{1-t}$, chosen to satisfy $G_0=id$, which is symplectic by taking the symplectic structure $\omega_{1-t}=\omega\circ G_t^{-1}$. Hence $(\xi_t,\omega_t)$ is a family of symplectic hyperplane bundles such that $\eta \subset \xi_0$. The situation before the first homotopy is pictured in Figure \ref{hom1}.

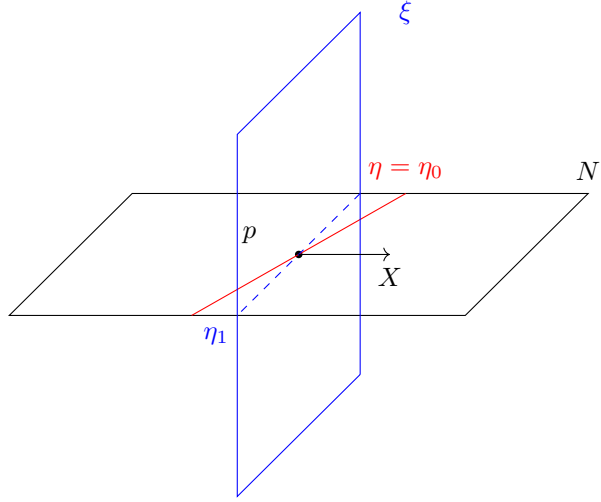
\begin{figure}[!h]
\centering
\begin{tikzpicture}[scale=6]

\draw (0,0,0)--(0,0,0.7)--(1,0,0.7)--(1,0,0)--cycle;

\node at (0.45,0.1,1/2) {$p$};
\draw (0.5,0,0.35) node [circle,fill, scale=0.3] {};

\draw[color=red] (0.6,0,0)--(0.4,0,0.7);
\node[color=red] at (0.6,0.05,0) {$\eta=\eta_0$};

\draw [color=blue, dashed] (0.5,0,0) -- (0.5,0,0.7);
\node[color=blue] at (0.5,0,0.82) {$\eta_1$};
\draw [color=blue] (0.5,-0.4,0)--((0.5,0.4,0)--(0.5,0.4,0.7)--(0.5,-0.4,0.7)--(0.5,-0.4,0);

\node[color=blue] at (0.6,0.4,0) {$\xi$};

\draw [->] (0.5,0,0.35)--(0.7,0,0.35);
\node at (0.7,-0.05,0.35) {$X$};

\node at (1,0.05,0) {$N$};

\end{tikzpicture}

\caption{Picture before first homotopy}\label{hom1}
\end{figure}

Assume that, if $2m \geq 3n-1$, any subbundle $\eta \subset (\xi_0, \omega_0)$ can be homotoped onto an isotropic one, i.e. any rank $n-1$ subbundle of a $2m$ dimensional symplectic bundle, over an $n$-dimensional manifold, is homotopic to an isotropic subbundle. This statement is the content of Lemma~\ref{isotr} below. In other words, we have a family of monomorphisms $F_t: \eta \to \xi_0$ such that $\eta_{is}:=F_1(\eta)$ is isotropic. We extend the monomorphisms $F_t$ into isomorphisms $H_t: \xi_0 \to \xi_0$ satisfying $H_0=id$, $F_t= H_t \circ F_0$. Clearly, the family of symplectic hyperplane bundles $(\xi_0, \omega_0\circ H_t)$ composed with the homotopy constructed in the previous paragraph, gives the required homotopy. The situation before this second homotopy is pictured in Figure~\ref{hom2}.
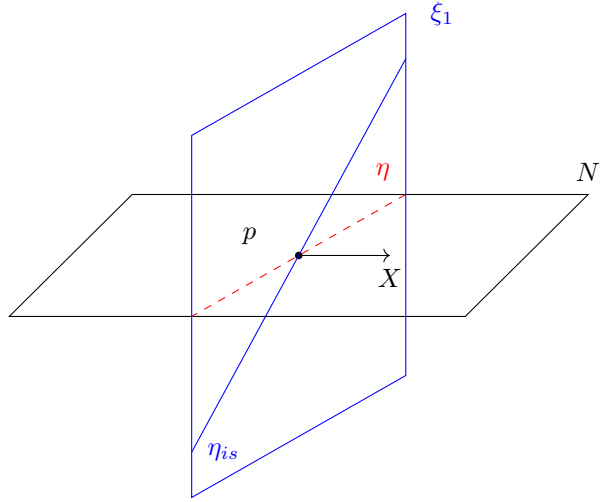
\begin{figure}[!h]
\centering
\begin{tikzpicture}[scale=6]

\draw (0,0,0)--(0,0,0.7)--(1,0,0.7)--(1,0,0)--cycle;

\node at (0.45,0.1,1/2) {$p$};
\draw (0.5,0,0.35) node [circle,fill, scale=0.3] {};

\node[color=red] at (0.55,0.05,0) {$\eta$};

\draw [color=red, dashed] (0.6,0,0) -- (0.4,0,0.7);
\draw [color=blue] (0.6,-0.4,0)--((0.6,0.4,0)--(0.4,0.4,0.7)--(0.4,-0.4,0.7)--(0.6,-0.4,0);

\draw [color=blue] (0.6,0.3,0)--(0.4,-0.3,0.7);
\node [color=blue] at (0.47,-0.3,0.7) {$\eta_{is}$};

\node[color=blue] at (0.68,0.4,0) {$\xi_1$};

\draw [->] (0.5,0,0.35)--(0.7,0,0.35);
\node at (0.7,-0.05,0.35) {$X$};

\node at (1,0.05,0) {$N$};

\end{tikzpicture}
\caption{Picture before second homotopy}\label{hom2}
\end{figure}

\end{proof}

\begin{lemma} \label{isotr}
{ Let $\xi$ be a symplectic bundle of rank $2m$ over $N$ and denote by $\eta$ a rank $n-1$ subbundle of $\xi$}. If $2m\geq 3n-1$, $\eta$ is homotopic to an isotropic subbundle.

\end{lemma}

\begin{proof}
Observe that we need to find a section of a bundle $E\rightarrow N$ whose fiber is $P=\operatorname{Path}(\operatorname{Grass}(n-1,\mathbb{R}^{2m}), \operatorname{Grass}_{is} (n-1,\mathbb{R}^{2m}))$, i.e. the space of paths connecting a fixed base point in $\operatorname{Grass}(n-1,\mathbb{R}^{2m})$ with end point in the Grassmanian of isotropic subspaces of dimension $n-1$ in $\mathbb{R}^{2m}$. This is a homotopy fibration with fiber homotopic to the space of loops in the $\operatorname{Gr}:=\operatorname{Grass}(n-1,\mathbb{R}^{2m})$ based on the subspace $\operatorname{Gr}_{is}:=\operatorname{Grass}_{is} (n-1,\mathbb{R}^{2m})$.

As explained in~\cite[Section 4.3, Proposition 4.64]{Hat} and the subsequent discussion, we have the identification $\pi_j(P) \cong \pi_{j+1}(\operatorname{Gr},\operatorname{Gr}_{is})$. By standard obstruction theory, a sufficient condition for the existence of such a section is to assume that the fiber $P$ is $(n-1)$-connected.

Recall that,
\begin{eqnarray*}
\operatorname{Gr} &  \cong  & \frac{SO(2m)}{SO(n-1) \times SO(2m-(n-1))}, \\
\operatorname{Gr}_{is} & \cong & \frac{U(m)}{SO(n-1)\times U(m-(n-1))}.
 \end{eqnarray*}
We have the following commutative diagram, given vertically by the relative exact sequences, and horizontally by quotients.
 \begin{center}
\begin{tikzcd}[column sep=tiny]
\pi_j(SO_{2m-n+2},U_{m-n+1}) \arrow{r}{} & \pi_j(SO_{2m},U_m) \arrow{r}{} & \pi_{j}(\operatorname{Gr},\operatorname{Gr}_{is}) \arrow{r}{} &  \pi_{j-1}(SO_{2m-n+2},U_{m-n+1}) \arrow{r}{} & \pi_{j-1}(SO_{2m},U_{m}) \\

\pi_j(SO_{n-1} \times SO_{2m-n+1}) \arrow{r}{} \arrow{u}{} & \pi_j(SO_{2m}) \arrow{r}{} \arrow{u}{} &\pi_j(\operatorname{Gr}) \arrow{r}{} \arrow{u}{}  & \pi_{j-1}(SO_{n-1}\times SO_{2m-n+1}) \arrow{r}{} \arrow{u}{} & \pi_{j-1}(SO_{2m}) \arrow{u}{}\\

\pi_j(SO_{n-1}\times U_{m-n+1}) \arrow{r}{} \arrow{u}{}& \pi_j(U_m) \arrow{r}{} \arrow{u}{} &\pi_j(\operatorname{Gr}_{is}) \arrow{r}{} \arrow{u}{}  & \pi_{j-1}(SO_{n-1}\times U_{m-n+1}) \arrow{r}{} \arrow{u}{}& \pi_{j-1}(U_m) \arrow{u}{}
\end{tikzcd}
\end{center}
We made the identifications $\pi_j(SO_{n-1}\times SO_{2m-n+1},SO_{n-1}\times SO_{2m-n+1}) \cong \pi_j(SO_{2m-n+2},U_{m-n+1}) \cong \pi_j(SO_{2m-n+2},U_{m-n+1})$, by using in the last isomorphism that we are in the stable range of $SO(2m-n+1)$.

Denote the gamma spaces $SO(2n)/U(n)$ by $\Gamma_n$. Observe that $SO(2n)$ is a fibration over $\Gamma_n$ with fiber $U(n)$. It is standard that the relative homotopy groups of a fibration with respect to the fiber are isomorphic to the homotopy groups of the base, see for instance \cite[Theorem 4.41]{Hat}. Hence we have the identification $\pi_j(SO(2n),U(n)) \cong \pi_j(\Gamma_n)$. In conclusion, the previous diagram is equivalent to the following one.

  \begin{center}
\begin{tikzcd}[column sep=tiny]
\pi_j(\Gamma_{m-n+1}) \arrow{r}{a_j} & \pi_j(\Gamma_m) \arrow{r}{} & \pi_{j-1}(P) \arrow{r}{} &  \pi_{j-1}(\Gamma_{m-n+1}) \arrow{r}{a_{j-1}} & \pi_{j-1}(\Gamma_{m}) \\

\pi_j(SO_{n-1} \times SO_{2m-n+1}) \arrow{r}{b_j} \arrow{u}{} & \pi_j(SO_{2m}) \arrow{r}{} \arrow{u}{c_j} &\pi_j(\operatorname{Gr}) \arrow{r}{} \arrow{u}{}  & \pi_{j-1}(SO_{n-1}\times SO_{2m-n+1}) \arrow{r}{} \arrow{u}{} & \pi_{j-1}(SO_{2m}) \arrow{u}{}\\

\pi_j(SO_{n-1}\times U_{m-n+1}) \arrow{r}{} \arrow{u}{}& \pi_j(U_m) \arrow{r}{} \arrow{u}{} &\pi_j(\operatorname{Gr}_{is}) \arrow{r}{} \arrow{u}{}  & \pi_{j-1}(SO_{n-1}\times U_{m-n+1}) \arrow{r}{} \arrow{u}{}& \pi_{j-1}(U_m) \arrow{u}{}
\end{tikzcd}
\end{center}
We want to prove that $\pi_{j-1}(P)$ is trivial up to $j-1=n-1$. To this end, let us show that we are in the stable range of $\Gamma_{m-n+1}$ up to rank $n-1$, and prove that $a_n$ is an epimorphism. The stable range of $\Gamma_{m-n+1} $ is $2(m-n+1)-2$, hence imposing that $n-1$ is in the stable range we obtain $n-1 \leq 2(m-n+1)-2$ which implies $ 3n-1 \leq 2m$, our dimensional hypothesis. Hence $a_j$ is an isomorphism for $j\leq n-1$ and we deduce $\pi_r(P)=0$ for $r\leq n-2$.

To conclude, observe that $\pi_n(\Gamma_{m-n+1})$ is in general no longer in the stable range. Let us check that $a_n$ is always at least an epimorphism, which will imply that $\pi_{n-1}(P)=0$. If $2m\geq 3n$, then we are in the stable range and $a_n$ is an isomorphism. If not, then $2m=3n-1$ and $n$ is odd. But the exact sequence induced by $\Gamma_k \rightarrow \Gamma_{k+1} \rightarrow \mathbb S^{2k}$, see~\cite{Gray}, at rank $n=2m-2n+1$ is the following.
$$ \pi_{2m-2n+1}(\Gamma_{m-n+1}) \rightarrow \pi_{2m-2n+1}(\Gamma_{m-n+2}) \rightarrow \pi_{2m-2n+1}(\mathbb S^{2m-2n+2}) $$
Since $\pi_{2m-2n+1}(\mathbb S^{2m-2n+2})=0$, the first arrow is an epimorphism. This implies that $a_n$ is always an epimorphism and the proof is complete.

%

\end{proof}



\subsection{Proof and discussion of Theorem \ref{main2}}\label{SS.fin}
We proceed with the proof of Theorem \ref{main2}, and a discussion of the result.

\begin{proof}[Proof of Theorem \ref{main2}]
Let $N$ be a compact manifold endowed with a geodesible field $X$. Denote by $e:(N,X)\rightarrow (M,\xi)$ an embedding into a contact manifold $(M,\xi)$. Let us assume that $M$ is overtwisted. Because of the codimension hypothesis, there is an homotopy $F_t:TN \rightarrow TM$ such that $F_0=de$,  $F_1(X) \pitchfork \xi$ and $\xi$ is positively transverse, i.e. by genericity it is needed $\text{dim}(N) < 2\,\text{dim}(M)$, which is clearly satisfied under our assumption $2\,\text{dim}(M)+1 \geq 3\,\text{dim}(N)$. Find isomorphisms $G_t:TM \rightarrow TM$ satisfying $F_t= G_t \circ F_0$. Define $\xi_t=G_t^{-1}(\xi)$ and define $\omega_t=d\alpha \circ G_t^{-1}$. It deforms $\xi$ to a formal contact structure $\xi_1$ satisfying that $F_0(X) \pitchfork \xi_1$.

Denote by $\eta=\ker \beta$ a transverse hyperplane distribution preserved by $X$. We can now apply Lemma~\ref{lemmaHo} { to our embedding with ambient symplectic bundle $(\xi_1,\omega_1)$ to} obtain a formal contact structure $(\tilde \xi,\tilde \omega)$ satisfying that $\tilde \xi \cap TN= \eta$ and, {that} $\eta$ is isotropic. Concatenating homotopies, we constructed a family of symplectic bundles $(\tilde \xi_t, \tilde \omega_t)$, $t\in [0,1]$, such that $(\tilde \xi_0, \tilde \omega_0)= (\xi, d \alpha)$ and $(\tilde \xi_1, \tilde \omega_1)=(\tilde \xi,\tilde \omega)$.

Since $\eta$ is isotropic, the bundle $\eta_\C=\eta \oplus \eta^*$ (endowed with the the standard complex structure) is a complex subbundle of $\tilde \xi_1$, and hence $\eta^*$ naturally lies, over $N$, in the normal bundle of $N$. The formal contact structure splits as $\tilde \xi_1= \eta_\C \oplus (\eta_\C)^{\perp}$ on a small tubular neighborhood of $N$, denoted as $\mathcal{O}p(N) \stackrel{pr}{\to} N$.
For a real constant $A$, take the homotopy of symplectic structures
$$ \tilde{\omega}_t= ((t-1)A+(2-t))\tilde \omega + (t-1) pr^*d\beta, \enspace t\in [1,2], $$
which will be a path of symplectic structures for a big enough $A>0$, as being symplectic is an open condition. We define $\tilde \xi_t=\tilde \xi_1$ for $t\in [1,2]$.
We obtain naturally $(\tilde \xi_t,\tilde \omega_t)$ for $t\in[0,2]$, a family of formal contact structures obtained by concatenating both homotopies. Clearly, we have that
\begin{equation}\label{dbeta}
\tilde \omega_2 \circ de =d \beta
\end{equation}
Now, as usual we undo the homotopy of contact structures by deforming the formal embedding. In order to do it, apply Corollary~\ref{coro:formalDarboux} to find a family of isomorphisms $\tilde G_t: TM \to TM$, $t\in [0,2]$, such that
\begin{itemize}
\item $\tilde G_t=G_t$ for $t\in [0,1]$,
\item $({\tilde G_t}^{-1}(\xi), d\alpha \circ {\tilde G_t})=(\tilde \xi_t, \tilde \omega_t)$ for $t\in [0,2]$.
\end{itemize}
 Thus, we define a family of monomorphisms  $\tilde{F}_t = \tilde G_t \circ F_0$ that satisfy $d\alpha \circ \tilde F_t = d\alpha \circ {\tilde G_t} \circ F_0= \tilde \omega_t \circ F_0$. For $t=2$, using equation (\ref{dbeta}), we have $\omega_2 \circ \tilde F_0=\omega_2 \circ de=d\beta$ and therefore it is a formal iso-Reeb embedding. We conclude applying Theorem \ref{hp1}.

Assume now that $M$ is not overtwisted, and hence $\dim M \geq 3\dim N +2$. Because of the dimension condition we can find an orthogonal symplectic decomposition $\xi|_N = \xi' \oplus L$, with $\operatorname{rank} L=2$ and for every $p\in N$ we have $L_p \cap TN_p=\{0\}_p$. We can assume this as long as $\dim M\geq 2\dim N +4$, which is true for $\dim N \geq 2$. { Hence there exists a family of symplectic bundles $\xi'_t$ such that $\xi_0'=\xi'$ and $X$ is transverse to $\xi_1'$}, and the proof applies verbatim by projecting $\eta$ into $\xi'$, which is a symplectic bundle of rank $2\,\dim M-2\geq 3\,\dim N-1$.
\end{proof}

Observe that, in fact, in Theorem~\ref{main2} we proved that for high enough codimension, any smooth embedding is isotopic to a (small) iso-Reeb embedding for any geodesible field and any invariant distribution. If we were to prove that our Theorem is sharp, we should find a geodesible field with a fixed invariant distribution on a manifold $N$ that does not admit an iso-Reeb embedding into a carefully fixed contact manifold of dimension $3\dim N-2$ or $3\dim N-1$ (depending on the parity of $\dim N$).

What we can prove is that there is a manifold of dimension $4k_0+1$ with a geodesible field and a fixed invariant distribution, and a smooth embedding of such a manifold into $\mathbb S^n$, where $n=3\dim N-4$, which is not deformable into an iso-Reeb embedding. So we are two dimensions away from the perfect sharpness.

\begin{prop} \label{lem:sharp}
There is a sequence of triples $(N_{k}, X_{k},\eta_k)$ of geodesible vector fields on a $k$-dimensional compact manifold $N_k$ with $k=4k_0+1$ such that there is no iso-Reeb embedding of $(N_k, X_k,\eta_k)$ into $(\mathbb S^{n},\xi)$ with $n < 3k-2$ and $\xi$ any contact structure.
\end{prop}

\begin{proof}
Let $W$ be a compact manifold such that $\dim W=4k_0$. Assume its Pontryagin classes $p_j(W)$ are all vanishing except the top one $p_{k_0}(W)$, which is non trivial (such as the manifolds constructed in~\cite{Miln}). Consider the manifold endowed with a geodesible vector field $(N=W \times \mathbb S^1, \partial_\theta)$ of dimension $k=4k_0+1$, with invariant $1$-form $d\theta$. The distribution $\eta=\ker d\theta$ is given by $TW$ seen as a distribution. If it admits a Reeb embedding into $(\mathbb S^n,\xi)$, we would have that $TW$ is an isotropic subspace of $\xi$. This follows from the fact that $d\theta$ is closed. Indeed, if we have a Reeb embedding $e$, there is a contact form $\alpha$ such that $e^*\alpha=d\theta$, so $e^*d\alpha=0$ and hence $d\alpha|_{TW}=0$.

Therefore, we have the decomposition $\xi|_N= TW^\mathbb{C} \oplus V$, where $V$ is the symplectic orthogonal to $TW^\mathbb{C}$. Using the Whitney sum formula for the total Chern class, we obtain that $0=c_{2k_0}(TW^\mathbb{C})+c_{2k_0}(V)$. Hence $V$ is of rank at least $4k_0$. This implies that $n\geq 8k_0+4k_0+1=3k-2$. For instance, $(\mathbb{C}P^2 \times \mathbb S^1, \partial_\theta)$ does not admit a Reeb embedding into $(\mathbb S^{11},\xi)$.
\end{proof}


\section{Final remarks}\label{S:end}

To conclude, let us make a few observations about the results in Section~\ref{S4} that are of independent interest. In Subsection~\ref{SS:ex} we provide some natural examples of iso-Reeb embeddings that appear in Contact Geometry, and in Subsection~\ref{topo} we analyze the topology of the moduli space of iso-Reeb embeddings, thus illustrating the wide range of iso-Reeb embeddings that our construction yields.

\subsection{Examples}\label{SS:ex}
In this section we give some additional examples of formal iso-Reeb embeddings:
\\

\noindent \emph{Formal isotropic $\eta$.}
Let $X$ be a geodesible vector field on $N$ with associated $1$-form $\beta$, and denote $\ker \beta=\eta$. Fix an embedding $e:N\rightarrow (M,\xi)$. Assume we can formally deform the embedding in such a way that $X$ is transverse to $\xi$ and $\eta$ is an isotropic subspace. Then perturbing the symplectic form as done in the proof of Theorem~\ref{main2}, we prove that it is a small formal iso-Reeb embedding.
\\

\noindent \emph{Totally isotropic embeddings.}
Consider an embedding $e: N \to (M, \xi)$ that is formal isotropic, we can actually make it isotropic by the $h$--principle for { isotropic embeddings or Lagrangian immersions \cite[Sections 12.4 and 16.1]{hprinc}}. So we assume that it is isotropic. Take any geodesible vector field $X$ on $N$ that preserves $\ker \beta=\eta$. We have the decomposition $TN=\langle X \rangle \oplus \eta$. Then, by the Weinstein neighborhood theorem $TM|_{N}= TN^{\C} \oplus V \oplus \langle R \rangle= \eta^\C \oplus \langle X, JX \rangle \oplus  V \oplus \langle R \rangle$, where $V$ is the symplectic orthogonal to $TN^{\C}$ inside $\xi$.
We claim that there is an arbitrarily small $C^{\infty}$-perturbation of the isotropic embedding that makes $X$ transverse and $\eta$ remains isotropic. The way of producing it is to flow the image $e(N)$ through the flow associated to $JX$. Do note that
$$\alpha(\mathcal L_{JX} X)= \alpha([JX, X])= d\alpha(JX, X)- X(\alpha (JX))- JX(\alpha (X))= d\alpha(X, JX) >0.$$
This shows that the image of $X$ through the flow becomes transverse to $\xi$. On the other hand, we obtain for any $Y\in \eta$ that, $\alpha(\mathcal L_{JX} Y)= 0$ and thus $\eta$ remains tangent to $\xi$. By perturbing the symplectic structure in $\eta^\C$ as in the proof of Theorem \ref{main2}, it is clear that it is a small formal iso-Reeb embedding.

\begin{Remark}\label{Emmy}
An alternative explanation of the last example was suggested to us by Emmy Murphy: apply the $h$-principle for isotropic immersions to make the embedding into an isotropic immersion {(by genericity, in this codimension we can assume that it is an embedding)}. There is a neighborhood of the embedding $\mathcal{O}p(N)$ contactomorphic  to a neighborhood of the zero section of the bundle $T^*N \times \R \times S$, where $S$ is a conformal symplectic bundle orthogonal to $T^*N$. The contactomorphism is provided by fixing the standard contact form $\alpha_{std} = dt - \lambda_{Liou}$ over $T^* N\times \R$, where $t \in \R$ and $\lambda_{Liou}$ is the canonical Liouville form in the cotangent bundle. Fix your geodesible vector field $(N, X, \eta= \ker \beta)$. There is a canonical embedding $\tilde{\beta}: N \to \mathbb S(T^*N) \subset T^* N \times \R$. By the universal property of the Liouville form, we have $\tilde{\beta}^* \lambda_{Liou}= \beta$. This implies, just by definition, that $\tilde{\beta}$ is a iso-Reeb embedding. In other words, if the vector field $X$ is geodesible on $N$, then it can be understood as the restriction of the geodesic flow on $\mathbb S(T^*N)$ and the geodesic flow is just the Reeb flow.

\end{Remark}

\subsection{The topology of the space of (small) iso-Reeb embeddings} \label{topo}

Finally, in this subsection, we compare the topology of the moduli space of iso-Reeb embeddings with the topology of the moduli space of smooth embeddings. To this end, we introduce some notation. For a compact manifold $N$ endowed with a geodesible field $X$, and target contact manifold $(M,\xi)$, we denote the space of iso-Reeb embeddings of $(N,X,\eta)$ into $(M,\xi)$ as $\mathfrak{Reeb}(N,M)$ and the space of formal iso-Reeb embeddings as $\mathcal{F}\mathfrak{Reeb}(N,M)$. Similarly we denote the space of small iso-Reeb embeddings as $\mathfrak{Reeb}_s(N,M)$ and the space of small formal iso-Reeb embeddings as $\mathcal{F}\mathfrak{Reeb}_s(N,M)$. In these last spaces we have made the notation minimal, since we should refer to  $(N,X,\eta,M,\xi)$ instead of $(N,M)$. Finally, denote by $\mathcal{S}(N,M)$ the space of smooth embeddings of $N$ in $M$. We have the following commutative diagram, where the maps are given by the natural inclusions.
\begin{center}
\begin{tikzcd}
\mathfrak{Reeb}(N,M) \arrow[hookrightarrow]{r}{i} & \mathcal{F}\mathfrak{Reeb}(N,M) \arrow[hookrightarrow]{d}{j}& &  \mathfrak{Reeb}_s(N,M) \arrow[hookrightarrow]{r}{i_s} & \mathcal{F}\mathfrak{Reeb}_s(N,M) \arrow[hookrightarrow]{d}{j_s} \\

& \mathcal{S}(N,M) & & & \mathcal{S}(N,M)
\end{tikzcd}
\end{center}

Using this notation, Theorem \ref{hp1} implies that if $(M,\xi)$ is overtwisted, and we only consider embeddings that do not intersect a fixed overtwisted disk, then $i$ is a homotopy equivalence. Theorem \ref{hp2} implies that $i_s$ is always a homotopy equivalence. In both cases we assume $\dim N < \dim M$. A parametric discussion of Theorem \ref{main2} implies that adding codimension is translated into isomorphisms in higher homotopy groups induced by $j$ and $j_s$.

\begin{Corollary}\label{smooth}
Let $N$ be a compact manifold endowed with a geodesible vector field $X$, and $(M,\xi)$ a contact manifold.
\begin{itemize}
\item If $\dim M > 3\dim N + 2 + k$ then $$j_s^r: \pi_{r}(\mathcal{F}\mathfrak{Reeb}_s(N,M)) \rightarrow \pi_r(\mathcal{S}(N,M)) $$ is an isomorphism for $r\leq k$.
\item If $M$ is overtwisted, $\dim M > 3\dim N+k$ and we consider embeddings not intersecting a fixed overtwisted disk, then $$ j^r: \pi_{r}(\mathcal{F}\mathfrak{Reeb}(N,M)) \rightarrow \pi_r(\mathcal{S}(N,M)) $$ is an isomorphism for $r\leq k$.
\end{itemize}
\end{Corollary}
\begin{Remark}
Observe that for $k=0$, we are increasing by $1$ the minimum codimension with respect to Theorem~\ref{main2}, this is because we are getting more. Theorem~\ref{main2} gives surjectivity of $j^0$ and here we obtain an isomorphism at the $\pi_0$ level.
\end{Remark}

\begin{proof}
Let us discuss the case where $M$ is overtwisted, the other case is analogous. The result follows from the proof of Theorem~\ref{main2}, which works parametrically by adding codimension. We want to prove that $j$ induces an isomorphism in homotopy groups up to rank $k$. To achieve this we only need that in the key step of Theorem~\ref{main2}, which is Lemma~\ref{lemmaHo}, the fiber $P=\operatorname{Path}(\operatorname{Grass}(n-1,\mathbb{R}^{2m}),\operatorname{Grass}_{is}(n-1,\mathbb{R}^{2m}))$ is $n+k$ connected. Following the notation and computations of Lemma~\ref{isotr}, we need that the space $\Gamma(m-n+1)$ is in the stable range up to rank $n+k$. Since the stable range is up to $2(m-n+1)-2$, we need to impose that $n+k\leq 2(m-n+1)-2$. This implies that $\dim M=2m+1>3\dim N+k$.
\end{proof}

By combining Corollary~\ref{smooth} with Theorem~\ref{hp2}, we deduce that we can replace $\mathcal{F}\mathfrak{Reeb}_s$ by $\mathfrak{Reeb}_s$ in Corollary \ref{smooth}, i.e. the isomorphisms of homotopy groups are between the spaces of genuine small iso-Reeb embeddings and smooth embeddings.


\begin{thebibliography}{99}

\bibitem{Ar65}
V. I. Arnold. \emph{Sur la topologie des \'ecoulements stationnaires des fluides parfaits}. C. R. Acad. Sci. Paris 261 (1965) 17--20.

\bibitem{AK} V. I. Arnold, B. Khesin. \emph{Topological Methods in Hydrodynamics} (Second Edition). Springer, New
York, 2021.

\bibitem{BEM} M. S. Borman, Y. Eliashberg, E. Murphy. \emph{Existence and classification of overtwisted contact structures in all dimensions}. Acta Math. 215 (2015) 281--361.

\bibitem{Bott} R. Bott. \emph{The stable homotopy of the classical groups}. Ann. of Math. 70 (1959) 313--338.

\bibitem{C1}  R. Cardona. \emph{Steady Euler flows and Beltrami fields in high dimensions}. Ergodic Theor. \& Dynam. Sys., 2020, 1-24. doi:10.1017/etds.2020.124.

\bibitem{CMP0} R. Cardona, E. Miranda, D. Peralta-Salas. \emph{Euler flows and singular geometric structures}. Phil. Trans. R. Soc. A 377, 20190034.

\bibitem{CMP2} R. Cardona, E. Miranda, D. Peralta-Salas. \emph{Turing universality of the incompressible Euler equations and a conjecture of Moore}. Int. Math. Res. Not. 2022 (2022) 18092--18109.
    
\bibitem{JMPA} R. Cardona, E. Miranda, D. Peralta-Salas. \emph{Computability and Beltrami fields in Euclidean space}. J. Math. Pures Appl. 169 (2023) 50--81.

\bibitem{CMPP2} R. Cardona, E. Miranda, D. Peralta-Salas, F. Presas. \emph{Constructing Turing complete Euler flows in dimension 3}. Proc. Natl. Acad. Sci. (2021), 118(19).

\bibitem{CPP} R. Casals, D. Pancholi, F. Presas. \emph{Almost contact 5-folds are contact}. Ann. of
Math. 182 (2015) 429--490.

\bibitem{Eng} R. Casals, J.L. Pérez, \'A. del Pino, F. Presas. \emph{Existence h-principle for Engel structures}. Invent. Math. 210 (2017) 417--451.

\bibitem{LEng} R. Casals,  A. del Pino, F. Presas. \emph{Loose Engel structures}. Compos. Math. 156 (2020) 412--434.


\bibitem{CE} R. Casals, J. Etnyre. \emph{Non-simplicity of isocontact embeddings in all higher dimensions}. Geom. Funct. Anal. 30 (2020) 1--33.

\bibitem{LS09}
C. De Lellis, L. Sz\'ekelyhidi. \emph{The Euler equations as a differential inclusion}. Ann. of Math. 170 (2009) 1417--1436.

\bibitem{El} Y. Eliashberg. \emph{Classification of overtwisted contact structures on 3-manifolds}. Invent. Math. 98 (1989) 623--637.

\bibitem{hprinc} Y. Eliashberg, N. Mishachev. \emph{Introduction to the h-principle}. AMS, Providence, RI, 2002.

\bibitem{ELP}
A. Enciso, R. Luc\`a, D. Peralta-Salas. \emph{Vortex reconnection in the three dimensional Navier–Stokes
equations}. Adv. Math. 309 (2017) 452--486.

\bibitem{EP1}
A. Enciso, D. Peralta-Salas. \emph{Knots and links in steady solutions of the Euler equation}. Ann. of Math. 175 (2012) 345--367.

\bibitem{EP2}
A. Enciso, D. Peralta-Salas. \emph{Existence of knotted vortex tubes in steady Euler flows}. Acta Math. 214 (2015) 61--134.

\bibitem{EG}J. Etnyre, R. Ghrist. \emph{Contact topology and hydrodynamics I. Beltrami fields and the Seifert
conjecture}. Nonlinearity 13 (2000) 441--458.

\bibitem{EG2} J. Etnyre, R. Ghrist. \emph{Contact topology and hydrodynamics III. Knotted orbits}.
Trans. Amer. Math. Soc. 352 (2000) 5781--5794.

\bibitem{FV93}
S. Friedlander, M. Vishik. \emph{Dynamo theory methods for hydrodynamic stability}. J. Math. Pures Appl. 72 (1993) 145--180.

\bibitem{Ge} H. Geiges. \emph{An Introduction to Contact Topology}. Cambridge Univ. Press, Cambridge, 2008.


\bibitem{Gl} H. Gluck. \emph{Dynamical behavior of geodesic fields}. Global theory of dynamical systems (Proc. Internat. Conf., Northwestern Univ., Evanston, Ill., 1979), pp. 190–215, Lecture Notes in Math., 819, Springer, Berlin, 1980.

\bibitem{Gray} J.W. Gray. \emph{Some global properties of contact structures}. Ann. of Math. 70 (1959) 313--338.

\bibitem{G} M. Gromov. \emph{Partial Differential Relations}. Ergebnisse der Mathematik und ihrer Grenzgebiete 9. Springer, Berlin, 1986.

\bibitem{Hat} A. Hatcher. \emph{Algebraic topology}. Cambridge Univ. Press, Cambridge, 2002.

\bibitem{Hi} M. W. Hirsch. \emph{Differential topology}. Springer-Verlag, New York, 1976.

\bibitem{Hus}D. Husemoller. \emph{Fibre Bundles}. Springer-Verlag, Berlin, 1994.

\bibitem{I} T. Inaba. \emph{Extending a vector field on a submanifold to a Reeb vector field on the whole contact manifold}. Unpublished paper \url{http://www.math.s.chiba-u.ac.jp/~inaba/} (2019).

\bibitem{Miln} M.A. Kervaire, J.W. Milnor. \emph{Bernoulli numbers, homotopy groups, and a theorem of Rohlin}. Proc. Internat. Congress Math. 1958.

\bibitem{Lee} J. M. Lee. \textit{Introduction to smooth manifolds}. Springer, New York, 2001.

\bibitem{DM} D. Mart\'inez Torres. \emph{Contact embeddings in standard contact spheres via approximately holomorphic geometry}. J. Math. Sci. Univ. Tokyo 18 (2011) 139--154.

\bibitem{Mur} E. Murphy. \emph{Loose Legendrian embeddings in high dimensional contact manifolds}. arXiv preprint arXiv:1201.2245 (2012).

\bibitem{P1} D. Peralta-Salas, A. Rechtman, F. Torres de Lizaur. \emph{A characterization of 3D steady Euler flows using commuting zero-flux homologies}. Ergodic Theor. \& Dynam. Sys. 41 (2021) 2166--2181.


\bibitem{Re}
A. Rechtman. \emph{Existence of periodic orbits for geodesible vector fields on closed $3$-manifolds}. Ergodic Theor. \& Dynam. Sys. 30 (2010) 1817--1841.

\bibitem{S} D. Sullivan. \emph{A foliation of geodesics is characterized by having no "tangent homologies"}. J. Pure
Appl. Algebra 13 (1978), no. 1, 101-104.

\bibitem{T0} T. Tao. \emph{Finite time blowup for an averaged three-dimensional Navier-Stokes equation}. J. Amer. Math. Soc. 29 (2016), no. 3, 601-674.

\bibitem{T1} T. Tao. \emph{On the universality of potential well dynamics}. Dyn. PDE 14 (2017) 219-238.

\bibitem{T2} T. Tao. \emph{On the universality of the incompressible Euler equation on compact manifolds}. Discrete Cont. Dyn. Sys. A 38 (2018) 1553--1565.

\bibitem{T3} T. Tao. \emph{On the universality of the incompressible Euler equation on compact manifolds, II. Nonrigidity of Euler flows}. Pure Appl. Funct. Anal. 5 (2020) 1425--1443.

\bibitem{TNat}
T. Tao. \emph{Searching for singularities in the Navier-Stokes equations}. Nature Reviews Physics 1 (2019) 418--419.

\bibitem{TdL} F. Torres de Lizaur. \emph{Chaos in the incompressible Euler equation on manifolds of high dimension}. Invent. math. 228, 687--715 (2022).


\end{thebibliography}
\end{document}